\documentclass{amsart}
\usepackage{amssymb,amsmath,amsthm,mathrsfs,graphics,hyperref,stmaryrd,psfrag,amscd,enumerate}
\numberwithin{equation}{section}
\newcommand{\N}{\mathbb{N}}
\newcommand{\Z}{\mathbb{Z}}
\newcommand{\Q}{\mathbb{Q}}
\newcommand{\R}{\mathbb{R}}
\newcommand{\C}{\mathbb{C}}
\newcommand{\excise}[1]{}

\DeclareMathOperator{\coker}{coker}
\DeclareMathOperator{\colspace}{Colspace}

\DeclareMathOperator{\rank}{rank}
\DeclareMathOperator{\rowspace}{Rowspace}
\DeclareMathOperator{\im}{Im}
\DeclareMathOperator{\lift}{lift}
\DeclareMathOperator{\Ten}{Ten}
\DeclareMathOperator{\Cut}{Cut}
\DeclareMathOperator{\Hom}{Hom}
\DeclareMathOperator{\Flow}{Flow}
\DeclareMathOperator{\zero}{zero}
\DeclareMathOperator{\sgn}{sgn}
\DeclareMathOperator{\supp}{supp}

\newcommand{\st}{\colon} 
\newcommand{\x}{\times}
\newcommand{\isom}{\cong}
\newcommand{\0}{\emptyset}
\newcommand{\sm}{\setminus}
\newcommand{\BB}{\mathcal{B}} 

\newcommand{\HH}{\mathcal{H}}
\newcommand{\MM}{\mathcal{M}} 
\renewcommand{\SS}{\mathcal{S}}
\newcommand{\tor}{{\mathbf T}}
\newcommand{\BBB}{B} 
\renewcommand{\H}{\tilde{H}} 
\newcommand{\Zz}{\Z}
\newcommand{\Nn}{\N}
\newcommand{\Qq}{\Q}
\newcommand{\Rr}{\R}
\newcommand{\bd}{\partial}
\newcommand{\cbd}{\partial^*}
\newcommand{\defterm}[1]{\emph{#1}}
\newcommand{\ori}{\varepsilon}

\newcommand{\includefigure}[3]{
  \begin{center}
  \resizebox{#1}{#2}{\includegraphics{#3}}
  \end{center}}

\renewcommand\emptyset{\varnothing}
\newcommand\calL{{\mathcal L}}

\newtheorem{theorem}{Theorem}[section]
\newtheorem{proposition}[theorem]{Proposition}
\newtheorem{lemma}[theorem]{Lemma}
\newtheorem{corollary}[theorem]{Corollary}
\theoremstyle{definition}
\newtheorem{definition}[theorem]{Definition}
\newtheorem{example}[theorem]{Example}
\newtheorem{remark}[theorem]{Remark}

\begin{document}

\author{Matthias Beck}
\address{Department of Mathematics, San Francisco State University}
\email{mattbeck@sfsu.edu}

\author{Felix Breuer}
\address{Department of Mathematics, San Francisco State University}
\email{felix@felixbreuer.net}

\author{Logan Godkin} 
\address{Department of Mathematics, University of Kansas}
\email{lgodkin@math.ku.edu}

\author{Jeremy L.\ Martin}
\address{Department of Mathematics, University of Kansas}
\email{jmartin@math.ku.edu}

\thanks{M.\ Beck was partially supported by the NSF through grant DMS-1162638.
F.\ Breuer was supported by the Deutsche Forschungsgemeinschaft (DFG), grant BR 4251/1-1.
J.L.\ Martin was partially supported by a Simons Foundation Collaboration Grant and NSA grant H98230-12-1-0274.}
\title{Enumerating Colorings, Tensions and Flows in Cell Complexes}
\keywords{graph, cell complex, chromatic polynomial, reciprocity, flows, tensions}
\subjclass[2000]{05A15, 05C15, 11P21, 52C35.}
\date{23 October 2013}

\begin{abstract}
We study quasipolynomials enumerating proper colorings, nowhere-zero
tensions, and nowhere-zero flows in an arbitrary CW-complex $X$,
generalizing the chromatic, tension and flow polynomials of a graph.
Our colorings, tensions and flows may be either
\emph{modular} (with values in $\mathbb{Z}/k\mathbb{Z}$ for some $k$)
or \emph{integral} (with values in $\{-k+1,\dots,k-1\}$). 
We obtain deletion-contraction recurrences and closed formulas
for the chromatic, tension and flow quasipolynomials, assuming certain
unimodularity conditions.  We use geometric methods, specifically
Ehrhart theory and inside-out polytopes, to obtain
reciprocity theorems for all of the aforementioned quasipolynomials, giving combinatorial interpretations of their values at negative integers as well as formulas for the numbers of acyclic and totally cyclic
orientations of $X$.
\end{abstract}

\maketitle
\section{Introduction}

This article is about generalizing the enumeration of
colorings, flows, cuts and tensions from graphs to cell complexes.
We begin with a review of colorings of graphs.

Let $G=(V,E)$ be a graph and let $k$ be a positive integer.  A
\emph{proper $k$-coloring} of~$G$ is a function $f:V\to K$, where $K$
is a ``palette'' of size~$k$ and $f(v)\neq f(w)$ whenever $vw$ is an
edge of~$G$.  It is well known that the number $\chi_G(k)$ of proper
$k$-colorings is a polynomial in~$k$ (the \emph{chromatic polynomial
  of $G$}).  Remarkably, the numbers $\chi_G(-k)$ have combinatorial
interpretations as well, as discovered by Stanley \cite{Stanley}; the
best known of these is that $|\chi_G(-1)|$ is the number of acyclic
orientations of~$G$.  This phenomenon is an instance of
\emph{combinatorial reciprocity} and is closely related to, e.g.,
Ehrhart's enumeration of lattice points in polytopes \cite{ehrhartpolynomial}
and Zaslavsky's theorems on counting regions of hyperplane
arrangements~\cite{Zaslavsky}.

Fix an orientation on the edges of $G$ and let $\bd$ be the signed
incidence matrix of $G$; that is, the rows and columns of~$\bd$ are
indexed by vertices and edges respectively, and the $(v,e)$ entry is
\[\begin{cases}
1 & \text{ if vertex $v$ is the head of edge $e$},\\
-1 & \text{ if vertex $v$ is the tail of edge $e$},\\
0 & \text{ if $v$ is not an endpoint of~$e$}.
\end{cases}\]
If we regard a coloring $f$ as a row vector $c=[f(v)]_{v\in V}$,
then properness says precisely that $c \, \bd$
is nowhere zero, since for each edge $vw$, the corresponding
entry of $c \, \bd$ is $f(v)-f(w)$.  Here the palette may be regarded either
as the integers $1,\dots,k$ or as the elements of $\Z/k\Z=\Z_k$.
The number of $k$-colorings may then be computed by linear algebra and
inclusion--exclusion, yielding Whitney's formula for the chromatic
polynomial (see, e.g.,~\cite[\S5.3]{West}).

The \emph{cut} and \emph{flow spaces} of~$G$ are respectively the row space and kernel
of its boundary matrix $\bd$, regarded as a map of modules over a ring $A$
(typically either $\Zz$, $\R$, or $\Z_k$).
Cuts and flows arise in algebraic graph theory
and are connected to the critical 
group and the chip-firing game: see, e.g., \cite{BDN} or \cite[Chapter~14]{AlgGraph}.
Over~$\Zz$, we consider in addition the space
of \emph{tensions}, or integer vectors of whom some multiple is
in the cut space.  (Over a field or a finite ring, all tensions are cuts.)
For a generalization of cuts and flows to cell complexes,
see \cite{DKM-CutFlow}.

If we regard the graph $G$ as a one-dimensional simplicial complex,
the matrix $\bd$ is just the boundary map from 1-chains to 0-chains.
Accordingly, we can replace the graph by an arbitrary $d$-dimensional
CW-complex~$X$ and define cellular colorings, flows, tensions and cuts
in terms of its top cellular boundary
map $\bd:C_d(X;\Zz)\to C_{d-1}(X;\Zz)$.

In Section~\ref{modular-section}, we study the functions
\begin{align*}
\chi^*_X(k) &~=~ \text{number of proper $\Zz_k$-colorings of $X$},\\
\tau^*_X(k) &~=~ \text{number of nowhere-zero $\Zz_k$-tensions of $X$},\\
\varphi^*_X(k) &~=~ \text{number of nowhere-zero $\Zz_k$-flows of $X$}.
\end{align*}
When $X$ is a graph, these are all polynomials in $k$ (for the
tension and flow polynomials, see \cite{Tutte}); in fact, they
are specializations of the Tutte polynomial of $X$.  For arbitrary cell
complexes, we show that the modular counting functions are always
quasipolynomials, and find sufficient
conditions on $X$ for them to be genuine polynomials
(building on the work of the first author and Y.~Kemper
\cite{yvonne}).  We describe two
avenues toward such results, using colorings as an example (the other
arguments are similar).  First, if $\bd$ contains a unit entry, then
pivoting the matrix $\bd$ there corresponds to the topological
operations of deleting a facet or deformation--retracting it onto a
neighboring ridge, and gives rise to a deletion--contraction recurrence
for modular colorings.  Provided that this ``shrinking'' process can
be iterated, it follows by induction that $\chi^*_X(k)$ is a
polynomial.  The second approach is a linear-algebraic
inclusion--exclusion argument, which produces a closed quasipolynomial
formula for $\chi^*_X(k)$ (Theorem~\ref{coloring-inc-exc}), which is
easily seen to be polynomial if every column-selected submatrix of
$\bd$ has no nontrivial invariant factors (a weaker condition than
total unimodularity).

In Section~\ref{integral-section}, we study the integral
counting functions
\begin{align*}
\chi_X(k) &~=~ \text{number of proper $K$-colorings of $X$},\\
\tau_X(k) &~=~ \text{number of nowhere-zero $K$-tensions of $X$},\\
\varphi_X(k) &~=~ \text{number of nowhere-zero $K$-flows of $X$},
\end{align*}
where $K$ is the palette $\{-k+1,-k+2,\dots,k-1\}\subseteq\Zz$. When $X$ is a graph, these are all polynomials in $k$ as was shown by Kochol~\cite{Kochol,KocholTension}. 
These can be regarded as Ehrhart functions, enumerating
lattice points in inside-out
polytopes~\cite{inout}.  While we cannot in general find
explicit combinatorial formulas for the integral counting functions,
we can interpret them combinatorially
in terms of acyclic and totally cyclic orientations of~$X$
(generalizing the graph-theoretic definitions), and obtain
reciprocity theorems: for instance,
$|\chi_X(-1)|$ is the number of acyclic orientations of~$X$.
Using these results, we obtain reciprocity theorems for the
modular counting functions in Section~\ref{modular-reciprocity-section},
by interpreting them as Ehrhart quasipolynomials of
certain modular analogues of inside-out polytopes.

The authors thank Hailong Dao, Yvonne Kemper and Jacob White
for numerous helpful suggestions.

\section{Background}

The symbol $\N$ will denote the nonnegative integers.  For $n\in\N$, we set $[n]:=\{1,\dots,n\}$,
and for $a \le b\in\Z$, we set $[a,b]:=\{a,a+1,\dots,b-1,b\}$.

We start by defining colorings, tensions and flows in the general
context of CW-complexes.  We assume familiarity with the basic
machinery of CW-complexes, including simplicial and cellular homology.
Unless otherwise stated, we will follow the notation and terminology
of Hatcher~\cite{Hatcher}.

Throughout the paper, $X$ will denote a cell complex of dimension~$d$.
We assume that $X$ is~\defterm{pure}, i.e., every maximal cell has the
same dimension.  (We lose no generality by this assumption; it is
analogous to assuming that graphs are connected.)
The symbols $F=\{f_1,\dots,f_m\}$ and
$R=\{r_1,\dots,r_n\}$ will always denote the sets of $d$-cells and
$(d-1)$-cells of~$X$; these are called \defterm{facets} and
\defterm{ridges}, respectively.  We allow the possibility that $F=\0$
(even though, strictly speaking, the dimension of such a complex is
less than~$d$).
For an Abelian group $A$, the chain groups $C_d(X;A)$ and $C_{d-1}(X;A)$
can be identified with the free modules $A^F$ and $A^R$ respectively.
The boundary map $C_i(X)\to C_{i-1}(X)$ is denoted $\bd_{A,i}$ or $\bd^X_{A,i}$, and 
the coboundary map $C_{i-1}(X)\to C_i(X)$ is $\cbd_{A,i}$. 
We simplify this notation by dropping the superscript when the complex~$X$ is clear from context;
moreover, if $A$ or $i$ (or both) are not specified, the reader should assume by default that $A=\Z$ and $i=d$.
The map $\bd$ is represented by an $n\x m$ matrix with columns indexed by~$F$ and rows by~$R$, and
entries $\bd_{rf}$ given by the degrees of the cellular attaching maps.
The (reduced) homology groups are
\[
  \H_i(X;A)=\ker\bd^X_{A,i}/\im\bd_{A,i+1} \, .  
\]

The \emph{cellular matroid} of $X$ is the matroid $\MM(X)$ represented by the
columns of $\bd_{\Rr,d}$.  (For a general reference on matroids, see, e.g., \cite{Oxley}.)
In the case that $X$ is a graph, the cellular matroid is just
the usual graphic matroid; more generally, a set $J\subset F$ is independent in $\MM(X)$
if $\H_d(X_J;\Zz)=0$ (where $X_J$ is the CW-subcomplex of~$X$ generated by $J\cup R$).
It is often convenient to use matroid language when speaking of
facets or sets of facets of $X$.  For instance,
a \emph{loop} of $X$ is a facet whose boundary is zero, and a \emph{coloop}
is a facet whose boundary is linearly independent of the boundaries of the other
facets.  

\begin{definition} \label{defn-coloring}
An \defterm{$A$-coloring} of $X$ is a vector $c=(c_1,\dots,c_n)\in
A^R$.  The coloring~$c$ is \defterm{proper} if $c \, \bd$ contains no
zero entries.  An \defterm{integral $k$-coloring} of~$X$ is a
$\Z$-coloring~$c$ such that $c_i\in[-k+1,k-1]$ for every~$i$.
\end{definition}

\begin{definition} \label{defn-flow}
The \defterm{$A$-flow space} of $X$ is the $A$-module
  $$\Flow_A(X) := \ker\bd_A \subseteq A^F = \H_d(X;A) \, .$$
An \defterm{integral $k$-flow} is a $\Zz$-flow $(w_1,\dots,w_m)$ such that
$w_i\in[-k+1,k-1]$ for every~$i$. In the graph setting, an enumeration theory of $k$-flows was initiated
by Kochol \cite{Kochol}; see also 
\cite{MR2274083,chenstanley}.
\end{definition}

\begin{definition} \label{defn-orient}
An \defterm{orientation} of $X$ is a sign vector $\ori\in\{
\pm 1\}^F$ that labels each facet of $X$ with a positive or negative
sign.
\end{definition}

\begin{example} \label{running-example}
Let $X$ be the boundary surface of the square pyramid on vertex set $\{1,2,3,4,5\}$, with base $2345$ and side triangles $123,134,145,152$, as shown.
\includefigure{1.5in}{1in}{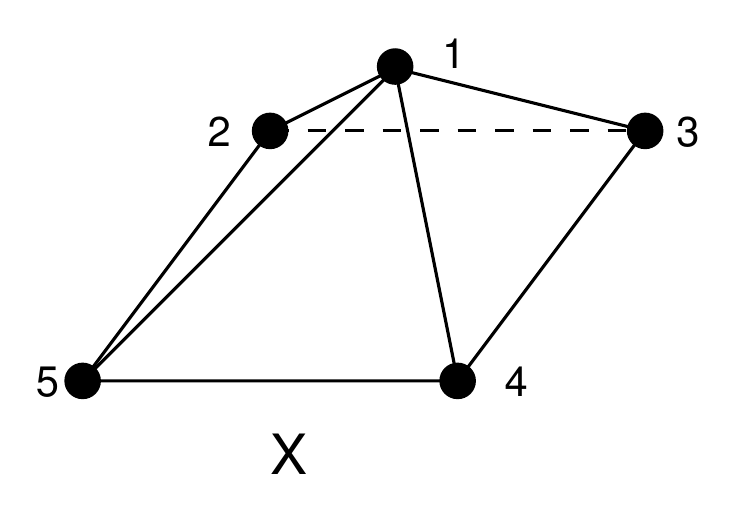} 
Orienting each facet clockwise from the outside, and each edge from smaller to greater endpoint, we obtain the cellular chain complex
\newcommand{\scs}{\scriptstyle}
\[
C_2 \xrightarrow[\bordermatrix{
 &\scs 123 &\scs 134 &\scs 145 &\scs 125 &\scs 2345\cr
\scs12 &\scs 1 &\scs 0 &\scs 0 &\scs 1 &\scs 0\cr
\scs13 &\scs-1 &\scs 1 &\scs 0 &\scs 0 &\scs 0\cr
\scs14 &\scs 0 &\scs-1 &\scs 1 &\scs 0 &\scs 0\cr
\scs15 &\scs 0 &\scs 0 &\scs-1 &\scs-1 &\scs 0\cr
\scs23 &\scs 1 &\scs 0 &\scs 0 &\scs 0 &\scs 1\cr
\scs25 &\scs 0 &\scs 0 &\scs 0 &\scs 1 &\scs-1\cr
\scs34 &\scs 0 &\scs 1 &\scs 0 &\scs 0 &\scs 1\cr
\scs45 &\scs 0 &\scs 0 &\scs 1 &\scs 0 &\scs 1}]{\bd_2}
C_1 \xrightarrow[\bordermatrix{
& \scs12 & \scs13 & \scs14 & \scs15 & \scs23 & \scs25 & \scs24 & \scs35 \cr
\scs1&\scs-1 &\scs-1 &\scs-1 &\scs-1 &\scs 0 &\scs 0 &\scs 0 &\scs 0\cr
\scs2&\scs1 &\scs 0 &\scs 0 &\scs 0 &\scs-1 &\scs-1 &\scs 0 &\scs 0\cr
\scs3&\scs0 &\scs 1 &\scs 0 &\scs 0 &\scs 1 &\scs 0 &\scs-1 &\scs-1\cr
\scs4&\scs0 &\scs 0 &\scs 1 &\scs 0 &\scs 0 &\scs 0 &\scs 1 &\scs 0\cr
\scs5&\scs0 &\scs 0 &\scs 0 &\scs 1 &\scs 0 &\scs 1 &\scs 0 &\scs 1}]{\bd_1}
C_0
\]
Note that $\partial_1$ is the signed incidence matrix of the graph that is the 1-skeleton of $X$. Only the top-dimensional boundary matrix $\partial_2$ is needed to define colorings, tensions, and flows of $X$.  For example, $c=[2,1,-3,-1,1,2,0,1]\in\Zz^8$ is a proper integral 4-coloring of $X$ since $c\partial_2=[2,4,-1,6,-1]$
is nowhere-zero; however, if we consider the entries of $c$ as elements of $\Zz_4$, then $c$ is not a proper $\Zz_4$-coloring since $c\partial_2\equiv[2,0,3,2,3]\mod 4$. The $\Zz$-flow space of $X$ is the kernel of $\bd_2$, the free $\Zz$-module of rank~1 generated by $[1,1,1,-1,-1]^T$,
and passing to $\Zz_4$ we see that $\Flow_{\Zz_4}(X)$ is generated over $\Zz_4$ by $[1,1,1,3,3]$ (in particular, $\Flow_{\Zz_4}(X)\isom\Zz_4$).

The cellular matroid of $X$ is uniform of rank~4 on ground set of size 5.  Linear-algebraically, every set of four columns of $\partial_2$ form a basis for its column space; topologically, no subcomplex generated by fewer than five facets (i.e., no proper subcomplex) has nonzero dimension-2 homology.
\end{example}

A flow $w\in \Flow_\Z(X)$ is \defterm{$\ori$-nonnegative} if
$\ori_i w_i \geq 0$ for every facet $f_i\in X$. The terms
\defterm{$\ori$-positive}, \defterm{$\ori$-negative}, etc., are
defined analogously. We write $\ori w$ to denote the vector obtained
by componentwise multiplication. A cell complex $X$ with given
orientation $\ori$ is \defterm{acyclic} if it has no nontrivial
$\ori$-nonnegative flows.  It is \defterm{totally cyclic} if for every
facet $f_i$, there exists a flow $w$ that is $\ori$-nonnegative and
$\ori_iw_i>0$.  An equivalent condition is that $X$ has a strictly
$\ori$-positive flow (since if $w^{(1)},\dots,w^{(m)}$ are
$\ori$-nonnegative flows with $\ori_i w^{(i)}_i>0$ for each $i$, then
$\sum_i \ori_i w^{(i)}$ is a strictly $\ori$-positive flow).

Let~$A^F$ be equipped with the standard inner product
$(a_f)_{f\in F}\cdot(b_f)_{f\in F}=\sum_{f\in F}a_fb_f$.  
For every $\Zz$-submodule $B\subseteq
A^F$, define $B^\perp=\{a\in A^F\st a\cdot b=0$ for every $b\in B\}$.

\begin{definition}\label{defn-cut-ten}
The \defterm{$A$-cut space} and \defterm{$A$-tension space} are
\begin{align*}
\Cut_A(X) &~=~ \rowspace_A\bd,\\
\Ten_A(X) &~=~ \Flow_A(X)^\perp.
\end{align*}
These are submodules of $A^F$; note that $\Ten_A(X)=(\Cut_A(X)^\perp)^\perp\supseteq\Cut_A(X)$
and that $\Ten_\Zz(X)/\Cut_\Zz(X)$
is isomorphic to the torsion summand of $\H^d(X;\Zz)$.
An \defterm{integral $k$-tension} of~$X$ is a $\Z$-tension~$\psi$ such that $\psi_i\in[-k+1,k-1]$ for every~$i$.
\end{definition}

\begin{example} 
Let $X$ be the square pyramid of Example~\ref{running-example}. The integral cut and tension spaces are isomorphic, since $X$ is homeomorphic to a 2-sphere and in particular has no torsion (co)homology.  Since $\rank_\Zz\Flow_\Zz(X)=1$, we have $\rank_\Zz\Cut_\Zz=5-1=4$, and looking at the generator of $\Flow_\Zz(X)$, we see that $\Cut_\Zz(X)=\{v_1,v_2,v_3,v_4,v_5\in\Zz\st v_1+v_2+v_3=v_4+v_5\}$, and likewise over $\Zz_4$.
The $\Zz_4$-coloring $c$ in Example~\ref{running-example} gives rise to the $\Zz_4$-tension/cut $\psi=c\bd_2=[2,0,3,2,3]$.
\end{example}

When $X$ is a graph, the numbers of nowhere-zero $\Zz_k$-cuts, flows and tensions are polynomial functions of $k$ \cite{Tutte} and the polynomials for cuts and tensions coincide. In the cell-complex setting,
we will see that they are quasipolynomials.
A \emph{quasipolynomial} is a function $q(k)$ of the form
$q(k) = c_n(k) \, k^n + c_{ n-1 }(k) \, k^{ n-1 } + \dots + c_0(k)$,
where $c_0(k), \dots, c_n(k)$ are periodic functions of $k$; the least common period of $c_0(k), \dots, c_n(k)$ is the
\emph{period} of $q(k)$.  (Thus a polynomial is a quasipolynomial
of period~1.)

\begin{proposition} \label{cuts-equal-tensions-mod-n}
Let $A$ be a field or a finite commutative ring.  Then $\Cut_A(X)=\Ten_A(X)$.
\end{proposition}

\begin{proof}
We will show that $(B^\perp)^\perp=B$, where $B$ is a submodule of a finitely generated free $A$-module $M=A^r$.
This is elementary if $A$ is
a field.  If $A$ is finite, then 
the inner product identifies $M^r$ with its Pontryagin dual
$M^*=\Hom(M,\Q/\Z)$.
(Usually the Pontryagin dual is defined as $\Hom(M,\C^\x)$, but we wish to write the group
law additively.) For any $A$-submodule (equivalently, Abelian subgroup) $B\subseteq M$, there is
a map $M^*\to B^*$, given by restriction, that is surjective (since any homomorphism $B^*\to\Qq/\Zz$
can be extended to a homomorphism $M^*\to\Q/\Z$) and has kernel $B^\perp$.  We therefore have
an isomorphism $M^*/B^*\isom B^\perp$; therefore $|M^*|/|B^*|=|M|/|B|=B^\perp$ (where the first equality follows from Pontryagin duality).  In particular $|(B^\perp)^\perp|=|B|$,
and since $B\subseteq (B^\perp)^\perp$ we must have equality.  Taking $B=\Cut_k(X)$, we obtain
the desired equality of the cut and tension spaces.
\end{proof}

The following useful fact is immediate from the foregoing definitions.

\begin{proposition} \label{coloring-cut}
The map $c\mapsto \psi := c\,\bd$ is a surjective $\Zz$-linear map from colorings to cuts.
The coloring $c$ is proper if and only if the cut $\psi$ is nowhere zero.
\end{proposition}

All of our definitions are natural extensions of the usual notions of cuts, flows and colorings for graphs.  A flow of a graph
is an element of its cycle space: an assignment of weights to the oriented edges so that the total flow into each
vertex equals the total flow out.  A coloring is a function on the vertices.
For graphs, cuts and tensions are synonymous: the minimal supports of nonzero elements of the cut space
are its bonds (the minimal edge sets whose deletion increases the number of components, or equivalently the
cocircuits of the graphic matroid).  The definitions of the cut and flow space of a cell complex are the same as those
of \cite{yvonne} and \cite{DKM-CutFlow} (with different notation).

\subsection{Unimodularity Conditions} \label{subsec:uni}

Let $M\in\Zz^{n\x m}$ be a matrix of rank~$\rho$.  We say that $M$ is
\defterm{totally unimodular} (TU) if and only if the determinant of
each square submatrix is $-1$, $0$, or $1$.  Totally unimodular
matrices arise in several contexts, such as integer linear programming
\cite[Chapter~8]{Schrijver} and as representations of regular matroids
\cite[Chapter~6]{Oxley}.  A \defterm{totally unimodular cell complex}
is one whose top boundary map is TU.  In topological terms, $X$ is TU
if and only if for all subcomplexes $Z\subseteq Y\subseteq X$ with
$\dim Y=d$ and $\dim Z=d-1$, the relative homology group
$\H_{d-1}(Y,Z;\Zz)$ is torsion free, as proved by Dey, Hirani and
Krishnamoorthy~\cite[Theorem~5.2]{OTL}.  For example, every graph is
TU.

The following facts about TU matrices may be well known, but
we have not found an explicit reference in the literature,
so we include short proofs here.  For a general reference on
TU matrices and their properties, see, e.g., \cite[Chapter 19]{Schrijver}.

\begin{lemma} \label{integer-solution}
Let $M$ be a $n\x m$ TU matrix and $B\in\Zz^n$.
Then the system of equations $MX=B$ has a rational solution in~$X$ if and only if it has
an integer solution.
\end{lemma}
\begin{proof}
The ``if'' direction is clear.  For the ``only if'' direction, let $X_0$ be a rational solution.
Then $X_0$ is a solution to the system of linear inequalities $AX\leq C$, where
\[A = \begin{bmatrix} M \\ -M \\ I \\ -I \end{bmatrix},\qquad
C = \begin{bmatrix} B \\ -B \\ P \\ -Q \end{bmatrix}\]
are matrices in block form, where $I$ is the $m\x m$ identity matrix and $P,Q$ are obtained by rounding the coordinates of~$X$ up and down, respectively, to the nearest integers.  Then $A$ is a TU matrix \cite[eqn.~(43), \S19.4]{Schrijver}, and the
solution space of $\{X\st AX\leq C\}$ is a nonempty polytope contained in the affine linear space $\{X\st MX=B\}$.
By \cite[Thm.~19.1]{Schrijver}, the vertices of P are all integral, so any such vertex is the desired solution.
\end{proof}

\begin{lemma}
\label{tu-basis-of-kernel}
Let $m\geq n$ and let $M$ be a $n\x m$ TU matrix.
Then there exists a $(m-n)\x m$
TU matrix $Z$ such that the rows of $Z$ are a $\Zz$-basis for
$\ker A$.
\end{lemma}

\begin{proof}
Write $M=[B|C]$.  Assume, without loss of generality, that $B$ is a column basis for $M$,
so that it is also a $\Zz$-basis for $\colspace_{\Zz}(M)$.
In particular, $B$ is invertible.
By applying repeated pivot operations,
we can transform the first $r$ columns into the identity
matrix; this is equivalent to multiplying $M$ on the
left by $B^{-1}$, so the resulting matrix is $[I|B^{-1}C]$.
By \cite[Lemma~2.2.20]{Oxley}, this matrix is TU.
Let $Z = [(B^{-1}C)^T|-I]\in\Zz^{(m-n)\x m}$.  Then the rows of $Z$ are a
vector-space basis for $\ker M$, and moreover $Z$ is TU because every
minor of $Z$ equals some minor of $[I|B^{-1}C]$.
\end{proof}

In order to state our enumerative results for as general a class
of cell complexes as possible, we will consider various weakenings of
the TU condition, as we now explain.

\begin{definition}
The matrix~$M$ is \defterm{quasi-unimodular} (QU) if the greatest common divisor of its maximal-rank minors
is 1.  It is \defterm{strongly quasi-unimodular} (SQU) if every column-selected submatrix of $M$ is QU.
\end{definition}

We recall some standard facts about finitely
generated Abelian groups.  There
exist unique positive integers $a_1,\dots,a_\rho$, the \emph{invariant factors} of~$M$,
such that $a_1|a_2|\cdots|a_\rho$ and
\[
  \coker M:=\Zz^n/\colspace(M)\isom \Zz^{n-\rho}\oplus\bigoplus_{i=1}^\rho\Zz_{a_i} \, .
\]
These numbers
can be computed as $a_i= \frac{ g_i }{ g_{i-1}}$, where~$g_i$ denotes the
gcd of the $i\x i$ minors of $M$ (putting $g_0=1$).  In particular the
invariant factors of a matrix and its transpose are identical.

Define a quasipolynomial
function of $M$ by
\begin{equation} \label{gammaMk}
\gamma(M,k)=\prod_{i=1}^{\rho} \gcd(k,a_i) \, .
\end{equation}
Knowing the function $\gamma(M,k)$ is equivalent to knowing the isomorphism
type of the group $\coker M$.  Moreover, if we
denote by $M\otimes\Zz_k$ the matrix $M$ with entries regarded as
elements of $\Zz_k$, then
\begin{equation} \label{count-kernel}
|\ker(M\otimes\Zz_k)| = k^{n-\rho} \gamma(M,k) \, .
\end{equation}
If $X$ is a cell complex of dimension~$d$, we put  $\gamma(X,k)=\gamma(\cbd,k)$,
where $\cbd=\cbd_d(X)$ is its top coboundary matrix.  Then
\begin{equation}\label{cohocount}
|\H^d(X;\Zz_k)|
=\frac{|(\Zz_k)^m|}{|\im(\cbd\otimes\Zz_k)|}
=k^m\frac{|\ker(\cbd\otimes\Zz_k)|}{|(\Zz_k)^n|}
=k^{m-\rho} \gamma(X,k) \, .
\end{equation}

The product $a_1a_2\cdots a_i$ equals the greatest common divisor of the
$i\x i$ minors of~$M$.  Thus the boundary matrix of
a CW-complex~$X$ is QU if and only if $\H_{d-1}(X;\Zz)$ is torsion free,
and it is SQU if and only if $\H_{d-1}(Y;\Zz)$ is torsion-free for every
subcomplex $Y\subseteq X$ generated by a subset of facets.

\subsection{Deletion and Contraction} \label{subsec:delcont}

Let $\bd$ be an $n\x m$ matrix over~$A$, and let $r\in[n]$ and $f\in[m]$.
Define a matrix $\bd'$ by deleting the $f^{th}$ column from $\bd$.
If $\bd_{rf}$ is a unit in~$A$, then we may in addition define
a matrix $\bd''$ by pivoting at $\bd_{rf}$:
\begin{itemize}
\item For each $i\neq r$, 
replace the row $\bd_i$ by $\bd_i-\frac{\bd_{if}}{\bd_{rf}} \, \bd_r$.
\item Delete the row and column corresponding to $r$ and $f$ respectively.
\end{itemize}
That is, the entries of $\bd''$ are
\begin{equation} \label{pivot}
\bd''_{ij} = \bd_{ij}-\frac{\bd_{if}}{\bd_{rf}} \, \bd_{rj} \, .
\end{equation}

If $\bd$ is the (integer) cellular boundary matrix of a cell complex~$X$
and $\bd_{rf}=\pm1$, then these operations carry topological meaning:
$\bd'$ and $\bd''$ are respectively the boundary matrices of the deletion $X'=X\sm f$,
and the ``contraction'' $X/rf=X/(r,f)$ obtained by deformation-retracting the face~$f$ onto the ridge~$r$.
For instance, if $X$ is a graph then these complexes
are the usual deletion and contraction of the edge~$f$.  More generally,
the column matroids of $\bd'$ and $\bd''$ are the deletion and contraction
of that of $\bd$~\cite[Proposition~3.2.6]{Oxley}.

\begin{example} 
Let $X$ be the square pyramid of Example~\ref{running-example}, shown again in~(a) below.  Contracting the triangular facet~123
onto the edge~13 gives the complex shown in~(b) below, with four facets and seven edges.  (Note that
we do not lose any vertices --- vertex~2 does not disappear, but moves to the line segment joining vertices~1 and~3.)
\includefigure{4in}{1in}{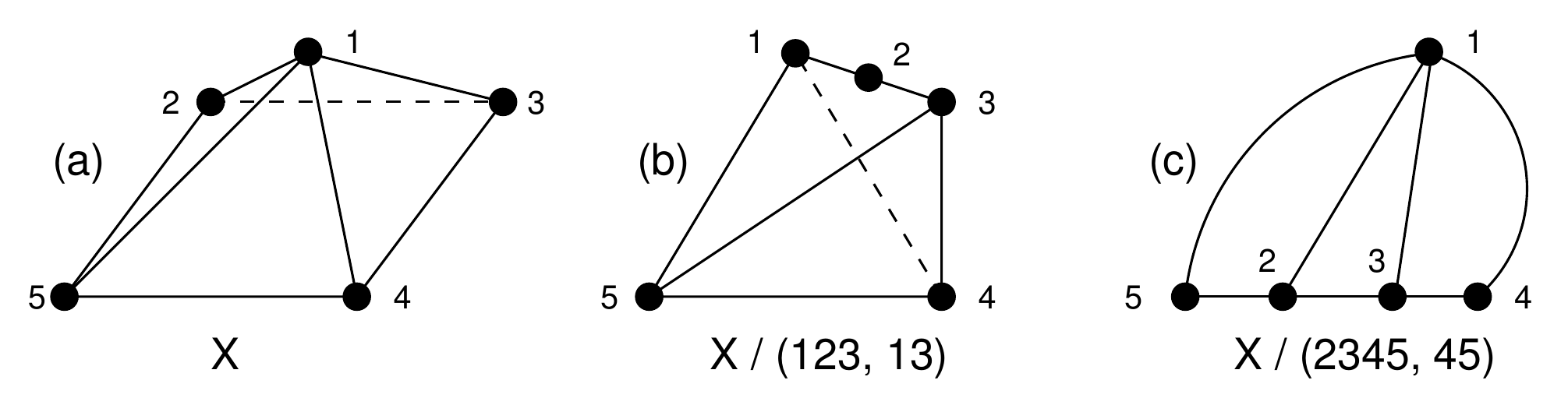} 
The boundary map $\bd_2''$ is given by pivoting on the starred entry of $\bd_2$:
$$\bd_2=
\begin{bmatrix}
 1 & 0 & 0 & 1 & 0\\
-1^* & 1 & 0 & 0 & 0\\
 0 &-1 & 1 & 0 & 0\\
 0 & 0 &-1 &-1 & 0\\
 1 & 0 & 0 & 0 & 1\\
 0 & 0 & 0 & 1 &-1\\
 0 & 1 & 0 & 0 & 1\\
 0 & 0 & 1 & 0 & 1
\end{bmatrix},\qquad
\bd_2''=
\begin{bmatrix}
 1 & 0 & 1 & 0\\
-1 & 1 & 0 & 0\\
 0 &-1 &-1 & 0\\
 1 & 0 & 0 & 1\\
 0 & 0 & 1 &-1\\
 1 & 0 & 0 & 1\\
  0 & 1 & 0 & 1
\end{bmatrix}.$$
The complex pictured in~(c) above is obtained from~$X$ by contracting the square facet onto edge~45.  (We omit the calculation of its boundary matrix.)
Geometrically, it consists of three triangles 152, 123, and~134, together with a pentagon with boundary~14325.
\end{example}

Total unimodularity of an integer matrix is preserved by pivot operations.  Therefore, any TU complex can be repeatedly contracted, choosing the facet and ridge arbitrarily at every step, to eventually yield a cell complex whose boundary map is either zero or empty.  More generally:

\begin{definition} \label{def:ISH}
Let $X$ be a cell complex and let $(r,f)$ be a ridge-facet pair.
Then $X$ is \defterm{shrinkable via $(r,f)$} if $\bd_{rf}=\pm1$.
It is
\defterm{iteratively shrinkable} (ISH) if $\bd$ is zero or empty,
or if it is shrinkable at some pair $(r,f)$, and $X\sm f$ and $X/rf$ are both iteratively shrinkable.
\end{definition}

Thus every TU complex is ISH; in fact, any non-loop facet in an ISH complex can
be chosen as the first one to shrink.  Moreover, every ISH complex is
QU, because if successively shrinking via the pairs $(r_1,f_1)\dots,(r_s,f_s)$ produces a zero or empty matrix,
then the submatrix with rows $r_1,\dots,r_s$ and columns $f_1,\dots,f_s$ has determinant $\pm1$.

To summarize, we have the implications
\[
\text{TU} \quad\implies\quad \text{SQU} \quad\implies\quad \text{QU}
\] and \[
\text{TU} \quad\implies\quad \text{ISH} \quad\implies\quad \text{QU}.
\]
The conditions SQU and ISH are incomparable.  For instance,
the matrix $\begin{bmatrix}3&2\\4&3\end{bmatrix}$ is invertible over $\Zz$,
hence SQU, but is not ISH because it contains no $\pm1$ entries.
On the other hand,
any matrix with
a full-rank identity submatrix is ISH but need not be SQU.

\subsection{Inside-out Polytopes}

A famous theorem of Stanley \cite{Stanley} gives a combinatorial
interpretation for the values of the chromatic polynomial
$\chi^*_G(k)$ of a graph $G$ at negative integers~$k$; in particular, $|\chi^*_G(-1)|$
is the number of acyclic orientations of $G$. Similar interpretations
have been found for the values of the integral and modular flow and
tension polynomials of graphs \cite{inout,MR2274083,breuersanyal}. In
Sections~\ref{integral-section} and~\ref{modular-reciprocity-section},
we will extend these results to cell complexes using lattice-point geometry,
specifically Ehrhart theory (see,
e.g., \cite{br}) and inside-out polytopes~\cite{inout},
which we review here.

Let $S$ be a subset of $\Rr^n$.  Its \defterm{Ehrhart function} is
\begin{equation}
\label{eqn:ehrhart}
E_S(k):=|\Zz^n\cap kS|
\end{equation}
where $kS$ denotes the $k^{th}$ dilate of~$S$ for $k\in\Z_{\geq 1}$.  If $Q$ is a rational convex polytope (i.e., the
convex hull of finitely many points in $\Rr^n$ with rational coordinates), then its Ehrhart function $E_Q(k)$ has the
following properties \cite{ehrhartpolynomial} (see also \cite[Theorem 3.23, Corollary 3.15]{br}):
\begin{itemize}
\item $E_Q$ is a quasipolynomial in~$k$ of degree $\dim Q$. In particular, this allows us to evaluate $E_Q$ at non-positive integers, even though \eqref{eqn:ehrhart} is only defined for positive integers.
\item The period of $E_Q$ divides the denominator of $Q$ (i.e., the
smallest $k$ such that $k^{-1}\mathbb{Z}^n$ contains every vertex of $Q$).
\item $E_Q(0)=1$.
\end{itemize}
We are going to be interested in the Ehrhart functions of various Boolean combinations of rational convex polytopes. Note
that for any such Boolean combination $S$, the above results automatically imply that $E_S(k)$ is a quasipolynomial and can
thus be evaluated at non-positive integers. Of special importance are the Ehrhart quasipolynomials of open polytopes: If
$Q^\circ$ is the relative interior of a rational convex polytope $Q$, then $E_{Q^\circ}(k)$ is a quasipolynomial and the
famous Ehrhart--Macdonald reciprocity theorem \cite{macdonald} (see also \cite[Theorem 4.1]{br}) states that 
\begin{equation}\label{ehrmacrec}
  (-1)^{\dim Q}E_{Q}(-k)=E_{Q^\circ}(k) \, .
\end{equation}

Let $\HH=\{H_1,\dots,H_m\}$ be a central hyperplane arrangement in~$\Rr^n$.  The corresponding \defterm{inside-out polytope} \cite{inout} is
$$C:=(-1,1)^n\sm\bigcup_{j=1}^m H_j.$$
Let $R_1^\circ,\dots,R_s^\circ$ be the open regions (i.e., connected components) of $C$,
and for each~$i$, let $R_i$ denote the topological closure of $R^\circ_i$.  Note that each  $R_i$ is a convex polytope and that $C$ is the disjoint union of the $R^\circ_i$. The Ehrhart quasipolynomial of $C$ is given by
$$E _{C}(k):=\sum_{i=1}^s E_{R_i^\circ}(k) \, .$$
We also define the \defterm{closed Ehrhart quasipolynomial} of~$C$ as
$$\bar{E}_{C}(k):=\sum_{i=1}^s E_{R_i}(k) \, .$$
For any positive integer $k$, the function $\bar{E}_C(k)$ counts lattice points $z\in(\Z^n\cap kC)$ with a multiplicity given by the number of regions $R_i$ that $z$ is contained in. Note that $E_C(0)=\sum_{i=1}^s E_{R_i}(0)=s$.
Our motivation for defining the closed Ehrhart quasipolynomial of $C$ is \eqref{ehrmacrec} which gives us
$$(-1)^{n}\bar{E}_{C}(-k)=E_{C}(k) \, .$$
Motivated by this reciprocity, we will sometimes call the Ehrhart quasipolynomial of $C$ the \defterm{open Ehrhart quasipolynomial} of $C$.

\section{Modular Counting Functions} \label{modular-section}

Throughout the rest of the paper, $X$ will denote a cell complex of dimension~$d$,
with facet and ridge sets $F,R$ of cardinalities $m,n$ respectively,
and top cellular boundary matrix $\bd$ of rank $\rho$.

In this section, we study the enumeration of $\Zz_k$-colorings,
tensions and flows on a cell complex~$X$.  The modular
chromatic, tension and flow functions are quasipolynomials
in~$k$.  Subject to certain unimodularity assumptions, they
satisfy deletion--contraction recurrences and Whitney-style
inclusion--exclusion formulas, and can be recovered (essentially)
from the Tutte polynomial of the cellular matroid of~$X$.
Many of the results generalize well-known results about the
chromatic, tension, and flow polynomials
of a graph; see, e.g., \cite[\S5.3]{West}.

\subsection{Modular Colorings} \label{subsec:mod-coloring}

Recall that an $A$-coloring of $X$ is a row vector $c=[c_r]_{r\in R}$
such that $c \, \bd$ is nowhere zero. More generally, a coloring of a
matrix $M$ is a row vector $c$ such that $cM$ is nowhere zero.

\begin{definition}
The \defterm{modular chromatic function} of~$X$ is
$$\chi^*_X(k):= \text{number of proper $\mathbb{Z}_k$-colorings of $X$}.$$
\end{definition}

When $X$ is a graph, this is just its chromatic polynomial.
We will show that in general, the modular chromatic function is a quasipolynomial, and
it is a polynomial when $X$ is SQU.

If $X$ has no facets, then $\chi_X^*(k)=k^n$ (because all colorings
are proper), while if $\bd$ has at least one zero column, then
$\chi^*_X(k)=0$.  (If $X$ is a graph, the second condition
says that $X$ has a loop.)  Moreover, if $P$ is any $n\x n$ matrix
that is invertible over~$\Zz$, then replacing $\bd$ with $P\bd$ does not change the
modular chromatic function of~$X$ (even though this operation does not in general
make topological sense).

\begin{theorem} \label{chromatic-del-cont}
Let $r$ and $f$ be a ridge and facet of $X$, respectively, such that $\bd_{rf}=\pm1$, and let $X'=X\sm f$ and $X''=X/rf$.
Then
$$\chi_X^*(k)=\chi^*_{X'}(k)-\chi^*_{X''}(k) \, .$$
\end{theorem}

\begin{proof}
Every proper coloring $c$ of $X$ is a proper coloring of $X\sm f$ (because if $c \, \bd=0$,
then $c \, \bd'_j=c \, \bd_j$ for all $j\in F\sm f$).  Conversely,
if $c'=(c_i)_{i\in R}$ is a proper coloring of $X'$ and $\psi=c \, \bd$, then
$\psi_j\neq 0$ for all $j\neq f$, and
$c$ is a proper coloring of~$X$ if and only if $\psi_f\neq0$.

Suppose that $\psi_f=0$.  Then we claim that the proper coloring $c'$ of $X'$ restricts to
a proper coloring $c''=(c_1,\dots,\widehat{c_r},\dots,c_n)$ of~$X''$ (where the hat denotes removal).
Indeed, for $j\neq f$, we have by~\eqref{pivot}
\begin{align*}
(c''\cdot\bd'')_j 
&= \sum_{i\in R\sm r} c_i \, \bd''_{ij}
= \sum_{i\in R\sm r} c_i\left(\bd_{ij}-\frac{\bd_{if}}{\bd_{rf}}\bd_{rj}\right)\\
&= \sum_{i\in R} c_i \, \bd_{ij}-c_r \, \bd_{rj}+\frac{\bd_{rj}}{\bd_{rf}}\left(c_r \, \bd_{rf}-\sum_{i\in R}c_i \, \bd_{if}\right)\\
&= \psi_j-\frac{\bd_{rj}}{\bd_{rf}}\psi_f ~\neq~ 0 \, .
\end{align*}

On the other hand, we claim that every proper coloring $c''=(c_1,\dots,\widehat{c_r},\dots,c_n)$ of~$X''$ can be extended uniquely to a proper coloring
$c'$ of $X'$ with $c \, \bd_f=0$.  Indeed, there is a unique value of $c_r$ for which $c \, \bd_f=0$, namely
$$c_r=-\frac{1}{\bd_{rf}}\sum_{i\in R\sm r} c_i \, \bd_{if} \, , $$
and then, for all $j\neq f$, we have (again using~\eqref{pivot})
$$c\cdot\bd_j = \sum_{i\in R}c_i \, \bd_{ij}
= \sum_{i\in R\sm r}c_i\left(\bd''_{ij}+\frac{\bd_{if}}{\bd_{rf}}\bd_{rj}\right)-\frac{\bd_{rj}}{\bd_{rf}}\sum_{i\neq r} c_i\bd_{if}
= \sum_{i\in R\sm r}c_i\bd''_{ij}
$$
which is nonzero by assumption on $c''$.  Accordingly, the correspondence between $c'$ and $c''$ is a bijection.
\end{proof}

\begin{remark}
This argument does not work for \emph{integral} $k$-colorings.  That
is, if $\bd_{rf}=1$ and $c\in\Zz^{R\sm\{r\}}$ is a proper $k$-coloring
of $X/rf$ with $-k<c_i<k$ for every $i$, then we may define
$c_r=\sum_{i\in R-\{r\}}c_i \, \bd_{if}$, but there is no reason to
expect that $-k<c_f<k$.
\end{remark}

We now consider the case that $f$ is a coloop of the cellular matroid.
Again, suppose that $\bd_{rf}=1$ and that $f$ and $r$ are represented by the first column and first row
of~$\bd$.
Pivoting at the $(1,1)$ position (a series of $\Zz$-invertible row operations) gives a matrix of the form
\[\begin{bmatrix}1&v\\ 0&\bd''\end{bmatrix}\]
in which $v$ is $\Qq$-linearly dependent on the rows of $\bd''$.  (Note that this condition is
always satisfied when $X$ is $SQU$.)
If $v$ is in fact a \emph{$\Zz$-linear}
combination of the rows of $\bd''$, then we can perform further $\Zz$-invertible row operations
to obtain the matrix
\[\begin{bmatrix}1&0\\0&\bd''\end{bmatrix}\]
whose colorings are precisely the row vectors of the form $[c_1,\dots,c_r]$
such that $c_1\in \Zz_k\sm\{0\}$ and $[c_2,\dots,c_r]$ is a coloring of $X''$.  Therefore, under these assumptions,
\begin{equation} \label{coloop-case}
\chi_X^*(k)=(k-1) \, \chi^*_{X''}(k) \, .
\end{equation}

We now consider how to express the modular chromatic function as a Whitney-style
sum
over subsets of facets, or as an evaluation of the Tutte polynomial.
For a set $J\subseteq F$ of facets, recall that $X_J$ is the CW-subcomplex of
$X$ generated by $J\cup R$. Let $\bd_J$ be its boundary matrix,
which is the submatrix of $\bd$ consisting of the columns indexed
by~$J$.  Let $\rho(J)=\rank\bd_J$; note that $\rho$ is the rank
function of the cellular matroid $\MM(X)$.

\begin{theorem} \label{coloring-inc-exc}
We have
\[\chi_X^*(k)=\sum_{J\subseteq F}(-1)^{|J|} k^{n-|J|} |\H^d(X_J;\Zz_k)|
=\sum_{J\subseteq F} (-1)^{|J|} k^{n-\rho(J)} \gamma(X_J,k)\]
where $\gamma$ is defined as in \eqref{gammaMk}.
In particular, the modular coloring function $\chi_X^*(k)$ is a quasipolynomial in~$k$
whose period divides the lcm of the set of all invariant factors of column-selected
submatrices of~$\bd$.
\end{theorem}

\begin{proof}
By inclusion--exclusion and \eqref{count-kernel},
\begin{align*}
\chi_X^*(k)&=\sum_{J\subseteq F}(-1)^{|J|}\left|\{c\in\Z_k^R:c \, \bd_J=0\}\right|\\
&=\sum_{J\subseteq F}(-1)^{|J|}\left|\ker_{\Zz_k}\cbd_J\right|\\
&=\sum_{J\subseteq F}(-1)^{|J|} k^{n-|J|} |\H^d(X_J;\Zz_k)|\\
&=\sum_{J\subseteq F}(-1)^{|J|} k^{n-\rho(J)} \gamma(\bd_J,k) \, ,
\end{align*}
proving the first assertion, from which the second follows as a special case.
The third assertion follows from the formula for the Tutte polynomial as a corank--nullity generating function \cite{BryOx}.
\end{proof}

\begin{corollary} \label{chromatic-from-Tutte}
Suppose that $X$ has no zero columns and $\bd$ is SQU (for instance, if $X$ is TU).
Then
\[\chi_X^*(k)=\sum_{J\subseteq F} (-1)^{|J|} k^{n-\rho(J)}
=(-1)^{\rho} k^{n-\rho} \, T_X(1-k, 0)\]
where $T$ is the Tutte polynomial of $\MM(X)$.
In particular, $\chi_X^*(k)$ is a monic polynomial in $k$ of degree $n$
that is divisible by $k^{n-\rho}$, whose coefficients alternate in sign,
and whose $k^{n-1}$-coefficient is the number of equivalence classes
of facets of $X$ under the relation that two facets are equivalent
if the corresponding columns of $\bd$ are scalar multiples of each other
over $\Qq$.
\end{corollary}

\begin{proof}
These assertions follow from the observation that 
$\chi_X^*(k)$ is just the characteristic polynomial of the hyperplane arrangement whose
normal vectors are the columns of~$\bd$; see, e.g., \cite[\S3.11.2]{EC2v2}.
\end{proof}

\begin{example}
Let $X$ be the real projective plane $\Rr P^2$ with its standard cell structure (one cell in each of dimensions 0, 1, 2).  The top boundary map
is $\bd=[2]$, so a modular $k$-coloring is an assignment of a color $c\in\Zz_k$ to the 1-cell such that $2c\neq 0\bmod k$.  Therefore,
the modular coloring function is
$$\chi_X^*(k)=\begin{cases} k-1 \text{ if $k$ is odd}\\k-2 \text{ if $k$ is even}\end{cases} = \ k-\gcd(k,2).$$
\end{example}

\begin{example}\label{dskeletonofNsimplex}
Let $X=\Delta_{N,d}$ be the $d$-dimensional skeleton of a simplex on $N$ vertices, so that $m=\binom{N}{d+1}$ and $n=\binom{N}{d}$.
This cell complex is QU (because $X$ is homotopy-equivalent to a wedge of $d$-dimensional spheres), but is not TU in general; for example, $\Delta_{6,2}$ contains
subcomplexes homeomorphic to $\Rr P^2$.  When $d=1$, $\Delta_{n,d}$ is a complete graph, whose chromatic polynomial is well known
to be $k(k-1)(k-2)\cdots(k-n+1)$.  A general formula for $\chi_X^*(k)$ in terms of $N$ and $d$ is not apparent.
\end{example}

\subsection{Modular Cuts and Tensions} \label{subsec:mod-cut-ten}

\begin{definition}
For $k\in\Nn$, let $\tau^*_X(k)$
be the number of nowhere-zero $\Zz_k$-tensions.
The function $\tau^*_X:\Nn\to\Zz$ is called the \defterm{modular tension function} of~$X$.
\end{definition}

\begin{proposition} \label{modular-cut-theorem}
We have
\[\tau^*_X(k) ~=~ \frac{k^{m-n}}{|\H^d(X;\Zz_k)|}\,\chi^*_X(k) ~=~ \frac{k^{\rho-n}}{\gamma(X,k)}\, \chi^*_X(k) \, .\]
\end{proposition}

\begin{proof}
Recall that the cut space and the tension space are identical over a finite base ring
(Proposition~\ref{cuts-equal-tensions-mod-n}).  Therefore,
the nowhere-zero tensions are precisely the vectors
of the form $c \, \bd$ for some proper coloring $c$, and each such tension
arises in this way from exactly $|\ker\cbd|$ colorings.
\end{proof}

By combining Proposition~\ref{modular-cut-theorem} with the results of
Section~\ref{subsec:mod-coloring}, we obtain the following ways to compute
$\tau^*_X(k)$.

\begin{corollary}
\begin{enumerate}[{\rm (1)}]
\item (Inclusion--exclusion formula) By Theorem~\ref{coloring-inc-exc}, we have
\begin{align*}
\tau^*_X(k) &= \frac{1}{|\H^d(X;\Zz_k)|}\,\sum_{J\subseteq F}(-1)^{|J|} k^{m-|J|} |\H^d(X_J;\Zz_k)|\\
&= \frac{1}{\gamma(X,k)}\, \sum_{J\subseteq F} (-1)^{|J|} k^{\rho-\rho(J)} \gamma(X_J,k) \, .
\end{align*}

\item (Deletion--contraction) Suppose that $\bd_{rf}=1$.  Let $X'=X-f$ and $X''=X/rf$.  Then
\begin{align*}
\tau^*_X(k)
&~=~ \frac{|\H^d(X' ;\Zz_k)|}{|\H^d(X;\Zz_k)|}\,k\,\tau^*_{X' }(k)
   - \frac{|\H^d(X'';\Zz_k)|}{|\H^d(X;\Zz_k)|}   \,\tau^*_{X''}(k)\\
~&=~ \frac{\gamma(X' ,k)}{\gamma(X,k)}\,k\,\tau^*_{X' }(k)
    -\frac{\gamma(X'',k)}{\gamma(X,k)}   \,\tau^*_{X''}(k) \, .
\end{align*}
If in addition $f$ is a coloop and the row vector $[\bd_{rj}]_{j\in F\sm\{f\}}$ is a $\Zz$-linear
combination of the vectors $[\bd_{ij}]_{j\in F\sm\{f\}}$ for $i\in R\sm\{r\}$, then \eqref{coloop-case}
implies that
\[\tau_X^*(k)=(k-1)\tau^*_{X''}(k) \, .\]

\item (Tutte-polynomial specialization) If $X$ is SQU and $T$ is the Tutte polynomial of $\MM(X)$, then
\[\tau^*_X(k) = (-1)^r T_X(1-k, 0) \, .\]
\end{enumerate}
\end{corollary}

\subsection{Modular Flows} \label{subsec:mod-flow}

\begin{definition}
For $k\in\Nn$, let $\varphi_X^*(k)$ denote
the number of nowhere-zero $\Zz_k$-flows on $X$.
The function $\varphi^*_X:\Nn\to\Zz$ is called the \defterm{modular flow function} of~$X$.
\end{definition}

\begin{theorem}
\label{flow-deletion-contraction}
Suppose that $X$ is shrinkable via the ridge-facet pair $(r,f)$.  Let $X'=X\sm f$ and $X''=X/rf$.
Then
$$\varphi_X^*(k)=\varphi^*_{X''}(k)-\varphi^*_{X'}(k) \, .$$
Moreover, if $w$ is a flow on $X$, then $w|_{X/rf}$ is a flow on $X/rf$ and if $w'$ is a flow on $X/rf$ then there exists a flow $w$ on $X$ with $w|_{X/rf}=w'$. 
\end{theorem}

\begin{proof}
Without loss of generality, assume that $\bd_{rf}=1$. 
By definition, a nowhere-zero $\Zz_k$-flow on the contraction $X''$ is a vector $(w_1,\dots,\widehat{w_f},\dots,w_m)$, with all $w_i\in\Zz_k\sm\{0\}$, such that
$$\sum_{i\neq f} w_i\bd_i = 0\mod\bd_f \, ;$$
that is, there is a unique scalar $w_f$ such that
$$w_f\bd_f = \sum_{i\neq f} w_i\bd_i \, .$$
If $w_f\neq 0$ then $(w_1,\dots,w_m)$ is a nowhere-zero flow on~$X$, while
if $w_f=0$ then $(w_1,\dots,\widehat{w_f},\dots,w_m)$ is a nowhere-zero flow on~$X'$.

Meanwhile, every nowhere-zero flow $w$ on $X$ or on $X'$ gives rise to a nowhere-zero flow on $X''$ (in the former case, by omitting $w_f$).
\end{proof}

If $X$ has no facets then $\varphi_X^*(k)=0$, and if $X$ has a loop $f$ then
\[
  \varphi_X^*(k) = k \, \varphi_{ X \setminus f }^* (k) \, . 
\]
Thus by induction on the number of facets, we have:
\begin{corollary}
If $X$ is ISH, then $\varphi_X^*(k)$ is a polynomial in~$k$.
\end{corollary}

\begin{theorem} \label{flow-inc-exc}
We have
\[\varphi_X^*(k)=\sum_{J\subseteq F} (-1)^{m-|J|} k^{|J|-\rho(J)} \gamma(X_J,k) \, .\]
In particular, the modular flow function $\varphi_X^*(k)$ is a quasipolynomial in~$k$
whose period divides the lcm of the set of all invariant factors of column-selected
submatrices of~$\bd$.

If $X$ has no zero columns and $\bd$ is SQU (for instance, if $X$ is TU), then
$$\varphi_X^*(k)=\sum_{J\subseteq F} (-1)^{m-|J|} k^{|J|-\rho(J)}.$$
In particular, in this case the modular flow function is a monic polynomial of degree~$m$, and
$$\varphi_X^*(k)=(-1)^{m-\rho} \,T_X(0,1-k)$$
where $T_X$ denotes the Tutte polynomial of $\MM(X)$.
\end{theorem}

\begin{proof}
Fix $k\in\Nn$.  By inclusion--exclusion and~\eqref{count-kernel},
\begin{align*}
\varphi^*_X(k)
&=\sum_{J\subseteq F}(-1)^{m-|J|} |\{w\in\ker_{\Zz_k}\bd \st w_f=0 \text{ for all $f\in F\sm J$}\}|\\
&=\sum_{J\subseteq F}(-1)^{m-|J|} |\ker_{\Zz_k}\bd_J|\\
&=\sum_{J\subseteq F}(-1)^{m-|J|} k^{|J|-\rho(J)} \gamma(X_J,k)
\end{align*}
as desired.  The third assertion follows from the formula for the Tutte polynomial as a corank-nullity generating function \cite{BryOx}.
\end{proof}

\begin{remark}
The columns of $\bd$ represent an \emph{arithmetic matroid}
in the sense of d'Adderio and Moci \cite[\S3.2]{arith-matroid}.
The corresponding \emph{arithmetic Tutte polynomial} is
\[M_X(x,y)=\sum_{J\subseteq F} \left| \tor(\H^d(X_J;\Zz)) \right| (x-1)^{\rho-\rho(J)} (y-1)^{|J|-\rho(J)}\]
where $\tor$ denotes the torsion summand.  In \cite{arith-colorings},
the same authors specialize the arithmetic Tutte polynomial to
enumerate colorings and flows on certain generalized edge-weighted
graphs.  Our modular chromatic, tension and flow functions do not
appear to be recoverable from $M_X$, because the
quasipolynomial $\gamma(X_J,k)$ specifies the actual
group structure of $\tor(\H^d(X;\Zz))$ rather than just its
cardinality.  On the other hand, the modular chromatic and tension
functions can be obtained from the Tutte--Grothendieck class of the
arithmetic matroid represented by~$\bd$ \cite[\S7]{MatRing}.
\end{remark}

\section{Integral Counting Functions and Reciprocity} \label{integral-section}

In this section, we extend various reciprocity results for integral counting polynomials defined on graphs to the more
general integral counting quasipolynomials defined on cell complexes. To this end, we use the general framework of
inside-out polytopes as developed by Beck and Zaslavsky \cite{inout,MR2274083}. In particular, the reciprocity theorem for
the integral chromatic quasipolynomial of a cell complex, Theorem~\ref{integral-coloring-reciprocity} below, generalizes Stanley's reciprocity theorem for the chromatic polynomial of a graph \cite{Stanley}. In this context, Proposition~\ref{regions-and-acyclic-orientations} generalizes Greene's geometric insight about acyclic orientations of graphs \cite{Greene}. The reciprocity theorem for the integral tension quasipolynomial, Theorem~\ref{integral-tension-reciprocity}, generalizes reciprocity results by Chen~\cite{Chen2010} and Dall~\cite{Dall}. Finally, the reciprocity theorem for the integral flow quasipolynomial, Theorem~\ref{integral-flow-reciprocity}, generalizes the integral flow reciprocity theorem by Beck and Zaslavsky~\cite{MR2274083}.

\subsection{Integral Colorings and Acyclic Orientations} \label{subsec:int-coloring}

Recall that an \defterm{integral $k$-coloring} of~$X$ is an element of~$(-k,k)^R\cap\Z^R$ (note that there
are $2k-1$ available colors for an integral $k$-coloring).  A coloring is \defterm{proper} if $c\, \bd$ is nowhere zero.

\begin{definition}
The \defterm{integral coloring function} of~$X$ is
$$\chi_X(k):= \text{number of proper integral $k$-colorings of $X$}.$$
If we regard a coloring $c$ as a lattice point in $\Zz^R$, then the proper $k$-colorings are precisely the
lattice points in $kC$, where $C$ is the inside-out polytope $(-1,1)^R\sm\HH_X$,
where $\HH_X$ is the arrangement of hyperplanes $\{H_f\st f\in F\}$ normal to the columns of $\bd$.
Therefore, the integral coloring function satisfies
\begin{equation} \label{coloring-Ehrhart}
\chi_X(k) = E_C(k) = (-1)^{n}\bar{E}_{C}(-k) \, .
\end{equation}
\end{definition}

\begin{definition}
Recall that an \defterm{orientation} of $X$ is a sign vector $\ori\in\{-1,1\}^F$.
We will abuse notation by also writing $\ori$ for the $m\x m$ diagonal matrix with entries $\ori_f$ for $f\in F$.
\defterm{Reorienting~$X$ by~$\ori$} means replacing $\bd$ with $\bd\ori$, i.e., multiplying each column of $\bd$ by the corresponding entry of~$\ori$.

Also recall that an orientation $\ori$ is \defterm{acyclic} if it admits no nontrivial nonnegative flows; that is, every nonzero element of
$\Flow(X,\ori)=\ker(\bd\ori)$ has at least one positive and at least one negative entry.
An equivalent condition is that $\bd\ori s\neq 0$ for every $s\in\Nn^F\sm\{0\}$.
\end{definition}

Let $c$ be an integral coloring and let $\ori$ be an orientation.  We say that the pair
$(c,\ori)$ is \defterm{compatible} if the vector $c \, \bd\ori$ has strictly positive entries, i.e.,
$\ori_f=1$ whenever $\sum_{r\in R} c_r \bd_{rf}>0$,
and $\ori_f=-1$ whenever $\sum_{r\in R} c_r \bd_{rf}<0$.
Every coloring has $2^a$ compatible orientations, where $a=|\{f\in F\st \sum_{r\in R} c_r \bd_{rf}=0\}|$.
In particular, a coloring is proper if and only if it has exactly one compatible orientation.

\begin{proposition}
\label{acyclic-orientations}
Let $c$ be a proper coloring of~$X$ and let $\psi=c \, \bd$ be the corresponding nowhere-zero tension (see
Proposition~\ref{coloring-cut}).   Define the orientation $\ori$ by
$$ \ori_f := \begin{cases} 1 & \text{ if } \psi_f>0,\\
-1 & \text{ if } \psi_f<0.\end{cases}$$
Then $\ori$ is acyclic.
\end{proposition}
\begin{proof}
It suffices to show that $\bd\ori s\neq 0$ for every $s\in\Nn^F\sm\{0\}$.
Indeed, for such $s$,
$$c \, \bd\ori s = \psi\cdot(\ori_fs_f)_{f\in F} = \sum_{f\in F} \psi_f\ori_fs_f \neq0$$
because $\psi_f\ori_f>0$ for every $f$, and it follows that $\bd\ori s\neq0$.
\end{proof}

\begin{proposition}
\label{regions-and-acyclic-orientations}
There is a bijection between regions of $\HH_X$ and acyclic orientations of~$X$.
\end{proposition}

\begin{proof}
Every open region $Q^\circ$ of $\R^R\setminus\bigcup_{f\in F} H_f$ is unbounded, hence contains some lattice point, which is a proper coloring
and therefore gives rise to an acyclic orientation by Proposition \ref{acyclic-orientations}.  Moreover,
all colorings in $Q^\circ$ give rise to the same acyclic orientation, denoted $\ori_Q$, in this way.  
If two regions are distinct, then they lie on opposite sides of at least one hyperplane~$H_f$, and so
the corresponding orientations will differ in (at least) their $f^{th}$ position.  Accordingly, it remains to show that
every acyclic orientation~$\ori$ of~$X$ has a compatible integral coloring.  This is a direct application of Gordan's
theorem from linear programming~\cite[Exercise~A.3, p.~334]{CombOpt} to the definitions of acyclic orientation and compatibility.
\end{proof}

\begin{theorem} \label{integral-coloring-reciprocity}
Let $X$ be an arbitrary cell complex.
The number of pairs $(\ori,c)$, where $\ori$ is an acyclic orientation of~$X$ and $c$ is a $k$-coloring
compatible with~$\ori$, equals $(-1)^{n}\chi_X(-k)$. In particular, the number of acyclic orientations of $X$ is
$(-1)^{n}\chi_X(-1)$.
\end{theorem}
\begin{proof}
Equation~\eqref{coloring-Ehrhart} may be rewritten as
$$\bar{E}_C(k)= (-1)^{n} E_C(-k)=(-1)^{n}\chi_X(-k)$$
where $C$ is the inside-out polytope $(-1,1)^n\sm\bigcup\HH_X$.
Meanwhile, $E_C(k)$ is the number of pairs $(R,c)$, where $c$ is a $k$-coloring and $R$ is a closed region of $\HH_X$ such that $c\in R$.  Furthermore, for each region $R$, the lattice points in $R^\circ$ are precisely the colorings compatible with the  orientation corresponding to~$R$, which implies the first assertion of the theorem.
The second assertion follows by setting $k=0$, because the only 1-coloring of $X$ is the zero vector,
which lies in every closed region.
\end{proof}

\subsection{Integral Tensions and Acyclic Orientations} \label{subsec:int-ten}

Recall that the tension space $\Ten(X)$ is the orthogonal complement
of the flow space $\Flow(X)=\ker\bd$.  In particular, the tension
space is a saturated $\Zz$-submodule of $\Zz^F$ (i.e., if
some nonzero scalar multiple of $\psi$ is a tension, then so is~$\psi$),
which is equivalent to the statement that $\Ten_\Zz(X)$ is a summand
of $\Zz^F$.

Recall also that an integral tension
(resp., a $k$-tension) is a tension $\psi\in\Zz^F$ (resp., a tension $\psi\in(-k,k)^{F}\cap\Zz^F$).

\begin{definition}
The \defterm{integral tension function} of~$X$ is
$$\tau_X(k):= \text{number of nowhere-zero $k$-tensions of $X$}.$$
\end{definition}

Let $\BB$ be the Boolean arrangement of coordinate hyperplanes in~$\Rr^F$. Define the inside-out polytope
$$T := \left((-1,1)^{F}\cap\Ten(X)\right)\sm \BB \, .$$
Then the nowhere-zero $k$-tensions are precisely the lattice points in~$kT$. 

If $X$ has a loop $f$, then $\Ten(X)$ is contained in the corresponding coordinate hyperplane.  In this case, there are no nowhere-zero
tensions; the inside-out polytope $T$ is empty; $\tau_X(k) = 0$; and~$X$ has no acyclic orientations.
We assume henceforth that $X$ has no loops (but note that Proposition
\ref{bijectiontensionregions} and Theorem \ref{integral-tension-reciprocity} below still hold trivially if $X$ has loops).

We have
\[
\dim(T) = \dim(\Ten(X)) = m - \dim(\Flow(X)) = \rho.
\]
Accordingly, the closed and open Ehrhart quasipolynomials of $T$
satisfy
\begin{equation} \label{tension-Ehrhart}
\tau_X(k) = E_{T}(k) = (-1)^{\rho}\bar{E}_{T}(-k)
\end{equation}
where the second equality comes from Ehrhart--Macdonald reciprocity. 

A tension $\psi$ and an orientation $\ori$ are called
\defterm{compatible} if, for all facets $f$, we have $\ori_f\psi_f\geq 0$.
That is, $\ori_f=1$ whenever $\psi_f>0$, and $\ori_f=-1$ whenever
$\psi_f<0$.  Note that every tension $\psi$ has $2^a$ compatible
orientations, where $a=|\{f\st \psi_f=0\}|$.  In particular, a tension
is nowhere zero if and only if it has exactly one compatible
orientation.

\begin{proposition}\label{bijectiontensionregions}
There is a bijection between regions of $\Ten(X)\sm\BB$ and acyclic orientations of~$X$.
\end{proposition}

\begin{proof}
Every open region $R^\circ$ of $\Ten(X)\sm\BB$ is unbounded, hence it contains some lattice point, which is a nowhere-zero tension and therefore gives rise to an acyclic orientation by Proposition~\ref{acyclic-orientations}.  Moreover,
all tensions in $R^\circ$ give rise to the same acyclic orientation~$\ori_R$ in this way.  If two
regions are distinct, then they lie on opposite sides of at least one coordinate hyperplane, and so the corresponding orientations will differ in that coordinate.  Accordingly, it remains to show that every acyclic orientation~$\ori$ of~$X$ has a compatible nowhere-zero tension. By Proposition~\ref{regions-and-acyclic-orientations}, every acyclic orientation $\ori$ has a compatible proper coloring $c$. Then by the definition of compatibility $\psi = c \, \bd_X$ is a nowhere-zero tension compatible with~$\ori$.
\end{proof}

\begin{theorem} \label{integral-tension-reciprocity}
Let $X$ be an arbitrary cell complex. The number of pairs
$(\ori,\psi)$, where $\ori$ is an acyclic orientation of~$X$ and
$\psi$ is a $k$-tension compatible with~$\ori$, equals
$(-1)^{\rho}\tau_X(-k)$. In particular, the number of acyclic
orientations of $X$ is $(-1)^{\rho}\tau_X(-1)$.
\end{theorem}

The argument is the same as that in the proof of Theorem~\ref{integral-coloring-reciprocity}.

\subsection{Integral Flows and Totally Cyclic Orientations} \label{subsec:int-flow}

Recall that a $k$-flow is an element of the flow space with integer entries strictly between $-k$ and $k$, i.e., a $k$-flow is an element $w\in\Flow(X)\cap(-k,k)^F\cap\Zz^F$.

\begin{definition}
The \defterm{integral flow function} of~$X$ is
$$\varphi_X(k):= \text{number of nowhere-zero $k$-flows of $X$}.$$
\end{definition}

As before, let $\BB$ be the Boolean arrangement of coordinate hyperplanes in~$\Rr^F$. Define the inside-out polytope 
$$W := \left((-1,1)^{F}\cap\Flow(X)\right)\sm \BB \, .$$
Then the nowhere-zero $k$-flows are precisely the lattice points in~$kW$. 

If $\bd_X$ has a column that is linearly independent from all others (i.e., the matroid represented by $X$ has a coloop), then $\Flow(X)$ is contained in the corresponding coordinate hyperplane.  In this case $W=\0$ and
$\varphi_X(k) = 0$, and $X$ has no totally cyclic orientations (and
Theorem~\ref{integral-flow-reciprocity} below is trivially true).
Henceforth, we assume that $X$ has no coloops, in which case
\[
\dim(W) = \dim(\Flow(X)) = m - \rho \, .
\]
Accordingly, the closed and open Ehrhart quasipolynomials of $W$ satisfy
\[
\varphi_X(k) =  E_W(k) = (-1)^{m-\rho}\bar{E}_W(-k)
\]
where the second equality comes from Ehrhart--Macdonald reciprocity.

Every lattice point in~$W$ gives rise to a totally cyclic orientation
(see Definition~\ref{defn-flow}), namely the vector
of signs of its coordinates.  Moreover,
lattice points in the same region of~$W$ give rise to the same totally cyclic orientation.  On the other
hand, for any totally cyclic orientation $\ori$, there is a strictly $\ori$-positive flow that is a lattice point
in $W$, and all $\ori$-positive flows belong to the same region of~$W$. In summary we obtain a result analogous to Proposition~\ref{bijectiontensionregions}.

\begin{proposition}\label{bijectionflowregions}
There is a bijection between regions of $\Flow(X)\sm\BB$ and totally cyclic orientations of~$X$.
\end{proposition}

We say that a $k$-flow $w$ and a totally cyclic orientation $\ori$ of a cell complex $X$ are \defterm{compatible} if $\ori w\geq 0$.

\begin{theorem} \label{integral-flow-reciprocity}
Let $X$ be an arbitrary cell complex. The number of pairs
$(\ori,w)$, where $\ori$ is a totally cyclic orientation of~$X$ and
$w$ is a $k$-flow compatible with~$\ori$, equals
$(-1)^{m-\rho}\varphi_X(-k)$. In particular, the number of totally cyclic
orientations of $X$ is $(-1)^{m-\rho}\varphi_X(-1)$.
\end{theorem}

Again, the argument is the same as that in the proof of Theorem~\ref{integral-coloring-reciprocity}.

\section{Modular Reciprocity} \label{modular-reciprocity-section}

In this section, we generalize various reciprocity theorems for modular counting polynomials defined in terms of graphs to
the case of modular counting quasipolynomials defined in terms of cell complexes. To this end, we follow the geometric
approach developed by Breuer and Sanyal~\cite{breuersanyal}, by realizing the counting functions as Ehrhart
quasipolynomials of a disjoint union of polytopes. This approach makes important use of the correspondence between regions
and orientations that we developed for the integral case in the previous section. In particular, the modular flow
reciprocity theorems, Theorem~\ref{modular-flow-reciprocity-theorem} and
Corollary~\ref{modular-flow-reciprocity-for-shrinkable-complexes} below, generalize the modular flow reciprocity theorem for graphs by Breuer and Sanyal~\cite{breuersanyal}. See Chen and Stanley~\cite{chenstanley} for a related reciprocity theorem. The reciprocity theorems for modular tensions, Theorem~\ref{modular-tension-reciprocity} and Corollary~\ref{modular-tension-reciprocity-for-shrinkable-complexes}, generalize reciprocity results by Chen~\cite{Chen2010} and Breuer and Sanyal~\cite{breuersanyal}. The reciprocity theorem for modular colorings, Theorem~\ref{modular-coloring-reciprocity}, generalizes the classic reciprocity for the chromatic polynomial by Stanley~\cite{Stanley}; see also Beck and Zaslavsky~\cite{inout}.

In order to obtain reciprocity theorems for the modular flow and tension
functions, we again interpret them as lattice points in inside-out polytopes, 
following the approach in \cite{breuersanyal}.

If $v=(v_1,\dots,v_k)$ is a vector, we write
\[\supp(v)=\{i\st v_i\neq0\},\qquad
  \zero(v)=\{i\st v_i=0\}.\]

\subsection{Modular Flow Reciprocity} \label{subsec:mod-flow-rec}

Recall that
the \defterm{modular flow function}
$\varphi_X^*(k)$ counts the nowhere-zero $\Z_k$-flows of $X$.
For an integer vector $b\in\Z^R$, define
\[
P_\bd(b):=\left\{w\in[0,1]^F:\bd w=b\right\},\qquad
P_\bd^\circ(b):=\left\{w\in(0,1)^F:\bd w=b\right\}.
\]
We call $b$ \emph{feasible} if $P_\bd^\circ(b)\neq\0$.  Let $\BBB=\BBB(\bd)$
denote the set of all feasible vectors, and define
\[
P_\bd:=\bigsqcup_{b\in \BBB}P_\bd(b) \, ,\qquad
P_\bd^\circ:=\bigsqcup_{b\in \BBB}P_\bd^\circ(b) \, .
\]
Then the modular flow function of $X$ is the Ehrhart function of $P_\bd^\circ$, that is,
$$\varphi_X^*(k) = E_{P_\bd^\circ}(k) = \left|\Z^F\cap k P_\bd^\circ\right|.$$  
Note that for all $b\in\BBB$ the dimension of $P^\circ_\bd(b)$ is $m-\rho$, since the intersection of a translate of $\Flow_R(X)$ with the open set $(0,1)^F$ has the same dimension as $\Flow_R(X)$. Therefore, we can apply Macdonald--Ehrhart reciprocity to obtain
\begin{align*}
(-1)^{m-\rho} E_{P^\circ_\bd}(-k)
&= \sum_{b\in \BBB} (-1)^{m-\rho} E_{P^\circ_\bd(b)}(-k) \\
&= \sum_{b\in \BBB}E_{P_\bd(b)}(k) \\
&= E_{P_\bd}(k)
\end{align*}
which gives us
\begin{equation}
\label{eqn:mod-flow-reciprocity}
(-1)^{m-\rho}\varphi_X^*(-k)
= \left|\Zz^F \cap kP_\bd \right|.
\end{equation}

Note that if $C$ is an $n\x n$ invertible integer matrix, then
\begin{align*}
\BBB(C^{-1}\bd)
&= \left\{b\in\Zz^R \st \{w\in(0,1)^F \st C^{-1}\bd w=b\}\neq\0\right\}\\
&= \left\{b\in\Zz^R \st \{w\in(0,1)^F \st \bd w=Cb\}\neq\0\right\}\\
&= C^{-1}\left\{a\in\Zz^R \st \{w\in(0,1)^F \st \bd w=a\}\neq\0\right\}\\
&= C^{-1}\BBB(\bd)
\end{align*}
and therefore
\begin{equation} \label{change-of-basis-flow}
P_{C^{-1}\bd}
= \bigsqcup_{b\in\BBB(C^{-1}\bd)} P_{C^{-1}\bd}(b)
= \bigsqcup_{a\in\BBB(\bd)} P_{C^{-1}\bd}(C^{-1}a)
= \bigsqcup_{a\in\BBB(\bd)} P_\bd(a) ~=~ P_\bd \, ,
\end{equation}
so we will henceforth drop the map $\bd$ from the notation.

\begin{example} \label{modular-flow-example}
Let $\bd=\begin{bmatrix}1&2\\2&4\end{bmatrix}$.  Then
\begin{align*}
P(0,0)&=\{(0,0)\} \, ,&
P(1,2)&=\calL\big((1,0),(0,\textstyle\frac12)\big) \, ,\\
P(2,4)&=\calL\big((1,\textstyle\frac12),(0,1)\big) \, ,&
P(3,6)&=\{(1,1)\} \, ,
\end{align*}
where $\calL(p,q)$ denotes the closed line segment from $p$ to $q$.
The feasible
vectors are $\BBB(\bd)=\{(1,2),(2,4)\}$.  The sets $P_\bd^\circ(b)$
are the thick line segments shown on the left in Figure~\ref{modular-flow-figure}. Note that these line segments are open, i.e., the endpoints are not included.
The number of interior lattice points in the $k^{th}$ dilate is
the quasipolynomial
\[
\varphi^*_X(k)=\begin{cases}
k-1 & \text{ if $k$ is odd},\\
k-2 & \text{ if $k$ is even}.
\end{cases}
\]
\begin{figure}[th]
  \includefigure{5.4in}{1.9in}{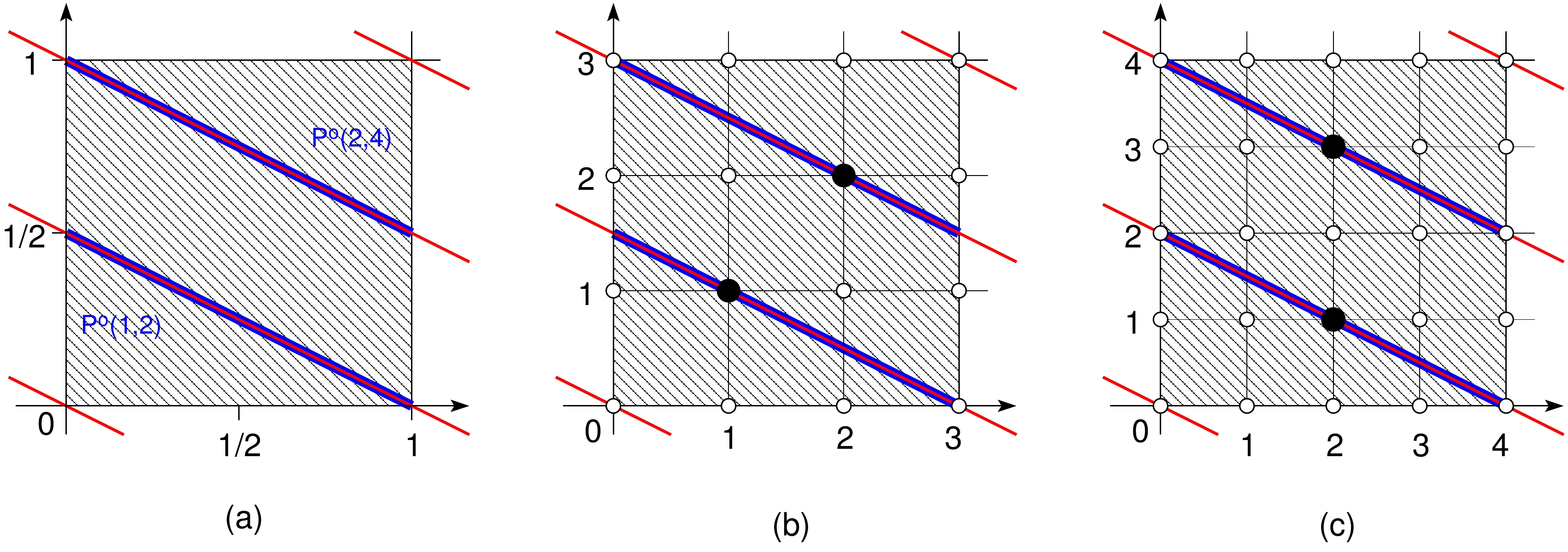} 
\caption{(a) The ``modular flow inside-out polytope'' of Example~\ref{modular-flow-example}.  (b), (c)  Its $k$-dilates for $k=3$ and $k=4$. The black dots
are the points of $P_\bd^\circ$, counted by $\varphi^*_X(k)$.\label{modular-flow-figure}}
\end{figure}
\end{example}

Recall from Section~\ref{subsec:int-flow} that $X$ is coloop-free
if and only if it has at least one totally cyclic orientation.
In addition, $X$ has no coloops if and only if $\ker\bd$ is not contained
in any coordinate hyperplane, i.e.,
it contains a nowhere-zero vector.  Indeed, if
$\bd$ has a nowhere-zero nullvector, then each column of~$\bd$ can be solved for
as a linear combination of the other columns.  Conversely, if $\bd$ has
no coloops, then for each facet~$f$ there is a nullvector with a nonzero
element in the $f^{th}$ position, and taking a generic real linear
combination of these nullvectors gives a nowhere-zero vector in
$\Flow_\Rr(X)$.

Our modular analogue of Theorem \ref{integral-flow-reciprocity} can be most easily expressed through 
\begin{equation} \label{flow-pair}
\Phi_k(X) := \left\{ (\bar{w},\sigma) \st
\begin{array}{l}
\text{$\bar{w}$ is a $\Z_k$-flow on $X$ and}\\
\text{$\sigma:\zero(\bar{w})\rightarrow\{-1,1\}$ extends}\\
\text{to a totally cyclic orientation of~$X$}
\end{array}
\right\}.
\end{equation}

\begin{theorem} \label{modular-flow-reciprocity-theorem}
Suppose that $X$ has no coloops.  Then
\[
(-1)^{m-\rho} \varphi_X^*(-k)=|\Phi_k(X)| \, .
\]
\end{theorem}

\begin{proof}
For $w\in\Z^F\cap k P_\bd$, define $h(w)=(\bar{w},\sigma)$, where $\bar{w}=w \bmod k$ and
\[
\sigma_f =
\begin{cases}
-1 & \text{ if } w_f = 0,\\
1 & \text{ if } w_f = k.
\end{cases}
\]
We will now show that $h$ is a bijection. By \eqref{eqn:mod-flow-reciprocity} this will complete the proof.

First, to show that $h$ is well-defined, we observe on the one hand that, by construction,
$\bar{w}$ is a $\Z_k$-flow on $X$ because $\bd\bar{w}=kb'$ for some $b'\in\Zz^R$. On the other 
hand, a totally cyclic orientation $\ori$ of $X$ with $\sigma = \ori|_{\zero(\bar{w})}$ can be found as follows. Let $b\in\Zz^R$ such that $w\in kP(b)$. By construction
\[
kP^\circ(b) = (\Flow_\R(X) + w)\cap (0,k)^F.
\]
In particular, $kP^\circ(b)$ is a relatively open, non-empty set of dimension $m-\rho$. Since $X$ has no coloops, the flow space is not contained in any coordinate hyperplane (that is, any hyperplane in the Boolean arrangement~$\BB$),
and the intersection of any coordinate hyperplane with $\Flow_\R(X)$ has dimension $m-\rho-1$. By the identity
\[
kP^\circ(b)\sm (\BB + w) = ((\Flow_\R(X) \sm\BB) + w)\cap (0,k)^F,
\]
it follows that $kP^\circ(b)\sm (\BB + w)$ can be obtained from $kP^\circ(b)$ by removing finitely many sets of dimension $m-\rho-1$. This implies in particular that $kP^\circ(b)\sm (\BB + w)$ is non-empty. Let $\hat{w}\in kP^\circ(b)\sm (\BB + w)$ and $v:=w-\hat{w}$. As both $w,\hat{w}\in\Flow_\R(X)+w$, it follows that $v\in\Flow_\R(X)$. As $\hat{w}\not\in(\BB+w)$, it follows that $v\not\in\BB$. Thus, $v$ is a nowhere-zero $\R$-flow on $X$. Let $\ori$ denote the sign vector of $v$. By the reciprocity theorem for integral flows (Theorem~\ref{integral-flow-reciprocity}) we know that $\ori$ is a totally cyclic orientation of $X$. Moreover, for $f\in\zero(\bar{w})$, either $w_f=0$, when $\hat{w}_f>w_f$, or else $w_f=k$, when $\hat{w}_f < w_f$, and in either case $\ori_f = \sigma_f$.

It is immediate from the definition that $h$ is injective.
To show surjectivity, suppose that $\bar{w}$ is a
$\Z_k$-flow on $X$ and $\sigma$ is a function
$\zero(\bar{w})\to\{-1,1\}$ such that there exists a totally cyclic
orientation $\ori$ of $X$ with $\sigma=\ori|_{\zero(\bar{w})}$.  
Define $w\in\Zz^F$ by
\[
w_f = \begin{cases}
\lift(\bar{w}_f) & \text{if $f\not\in\zero(\bar{w})$},\\
0 & \text{if $f\in\zero(\bar{w})$ and $\sigma_f=-1$},\\
k & \text{if $f\in\zero(\bar{w})$ and $\sigma_f=1$},
\end{cases}
\]
where $\lift(x)$ denotes the unique integer in $\{1,2,\dots,k-1\}$
that reduces to $x$ modulo~$k$.  Then $\bd w = kb$ for some $b\in\Z^R$
(because $\bar{w}$ is a $\Zz_k$-flow) and $w\in\Z^F\cap[0,k]^F$. So
all that remains to show is that $P_\bd^\circ(b)\neq\emptyset$.  By
Theorem~\ref{integral-flow-reciprocity}, there exists a nowhere-zero
real flow $v$ with $\sgn(v)=\ori$; we may choose $v$ to have
arbitrarily small magnitude.  We claim that $\hat{w}=w-v\in
P_\bd^\circ(b)$.  Indeed, first, we have
\[
  \bd\hat{w} = \bd w - \bd v = kb - 0  = kb \, .
\]
Second, if $f\not\in\zero(\bar{w})$ then $0<\hat{w}_f <k$.
If $w_f = 0$ then $\sgn(v_f)=-1$ and  $\hat{w}_f > 0$,
while if $w_f = k$ then $\sgn(v_f)=1$ and $\hat{w}_f < k$.
This proves the claim and completes the proof that $h$ is surjective.
\end{proof}

\begin{corollary}
If $X$ has no coloops, then $|\varphi_X^*(-1)|$
is the number of  totally cyclic orientations of $X$.
\end{corollary}
\begin{proof}
If we set $k=1$ in the definition \eqref{flow-pair},
we see that the vector $\bar{w}$ must be the zero vector, and $\sigma=\ori$ is
a totally cyclic orientation of~$X$.
\end{proof}

\begin{remark} \label{flow-coloops}
Suppose that $X$ has a coloop $f_1$.
Let $C$ be an invertible $r\x r$ matrix with an initial set of columns that form
a $\Zz$-basis for the saturation of $\Zz\{f_2,\dots,f_{m}\}$.
Then $C^{-1}\bd$ has the form
\[\left[\begin{array}{c|c}
a_1&\\
\vdots & \hat\bd\\
a_{\rho-1}&\\ \hline
a_\rho & 0\\ \hline
0&0
\end{array}
\right].\]
Note that $\hat\bd$ is the boundary map of the complex $X/f$ obtained by collapsing the cell $f$ to a point, and $w=(w_1,\dots,w_m)$ is a flow on $X$ if and only if
$(w_2,\dots,w_m)$ is a flow on $X/f$ and $w_1a_i\equiv 0 \bmod k$ for all $i=1,2,\dots,\rho$.  The number of values of $w_1$ for which the second condition holds is $\gcd(a_1,\dots,a_\rho,k)$.  Therefore,
$$\varphi^*_X(k) = \gcd(a_1,\dots,a_\rho,k)\, \varphi^*_{X/f}(k) \, .$$
\end{remark}

Note that if $X$ is a TU complex, then for any set $S$ of facets the contraction $X/S$ is well defined. In this case the modular flow reciprocity theorem for complexes can be phrased in exactly the same way as the modular flow reciprocity theorem for graphs.

\begin{corollary}
\label{modular-flow-reciprocity-for-shrinkable-complexes}
Let $X$ be a totally unimodular complex without coloops. Then $|\varphi^*_X(-k)|$ counts the number of pairs $(\bar{w},\sigma)$ where $\bar{w}$ is a $\Z_k$-flow on $X$ and $\sigma$ is a totally cyclic orientation on $X/\supp(\bar{w})$.
\end{corollary}

\begin{proof} To derive the corollary from Theorem~\ref{modular-flow-reciprocity-theorem}, all we have to show is that for any $\Z_k$-flow $\bar{w}$ and any $\sigma:\zero(\bar{w})\rightarrow\{-1,1\}$ the following statements are equivalent:
\begin{enumerate}[(1)]
\item $\sigma$ extends to a totally cyclic orientation of $X$. \label{enum:contract-support-1}
\item $\sigma$ is a totally cyclic orientation of $X/\supp(\bar{w})$. \label{enum:contract-support-2}
\end{enumerate}

Suppose that $\supp(\bar w)=\{f_1,\dots,f_\ell\}$.
By the observation after Definition~\ref{def:ISH},
we can find a sequence of
ridges $r_1,\dots,r_\ell$ such that the contractions
$X_0=X$, $X_1=X_0/(r_1,f_1)$, \dots, $X_i=X_{i-1}/(r_i,f_i)$, \dots,\ $X_\ell=X/\supp(\bar w)$
are well-defined.

\eqref{enum:contract-support-1} \emph{implies} \eqref{enum:contract-support-2}: Let $\ori$ be a totally cyclic orientation of $X$ such that $\ori|_{\zero(\bar{w})}=\sigma$. Then by Proposition~\ref{bijectionflowregions} there exists an $\ori$-positive $\R$-flow $v$ on $X$. We claim that $v|_{\zero(\bar{w})}$ is a $\sigma$-positive flow on $X/\supp(\bar{w})$. Clearly, $v|_{\zero(\bar{w})}$ is $\sigma$-positive.  By Theorem~\ref{flow-deletion-contraction} we know that if $v|_{F\setminus f_1,\ldots f_{i-1}}$ is an $\R$-flow on $X_{i-1}$, then $v|_{F\setminus f_1,\ldots f_i}$ is an $\R$-flow on $X_i$. By induction we obtain that $v|_{\zero(\bar{w})}$ is an $\R$-flow on $X_\ell=X/\supp(\bar{w})$. Thus, by Proposition~\ref{bijectionflowregions}, we conclude that $\sigma$ is a totally cyclic orientation on $X/\supp(\bar{w})$.

\eqref{enum:contract-support-2} \emph{implies} \eqref{enum:contract-support-1}: Let $\sigma$ be a totally cyclic orientation of $X/\supp(\bar{w})$. Then there exists a $\sigma$-positive $\R$-flow $v_\ell$ on $X_\ell=X/\supp(\bar{w})$. By Theorem~\ref{flow-deletion-contraction}, we can inductively construct $\R$-flows $v_\ell,\dots,v_0=v$ on $X_\ell,X_{\ell-1},\dots,X_0=X$ such that $\sgn(v_i|_{\zero(\bar{w})}) = \sigma$.  (Note that these flows need not be nowhere-zero.)

We now construct an $\R$-flow $w$ on $X$ such that $\zero(w)=\zero(\bar{w})$.  First, interpret $\bar{w}$ as a vector in $\Z^F$ by lifting all entries to their canonical representatives in $\Z$. Then $\bd\bar{w}=kb$ for some $b\in\Z^R$. Second, choose a lattice point $z\in\Z^F$ such that
\[z|_{\zero(\bar{w})} = 0 \quad\text{and}\quad \bd z = b \, .\]
(Such a lattice point must exist because this linear system of equations has a rational solution, namely $\frac{1}{k}\bar{w}$, and by Lemma~\ref{integer-solution} it must also have an integral solution.)
Finally, let $w = \bar{w} - kz$, so that $\bd w = 0$. Note that $w_f\not=0$ whenever $f\in\supp(\bar{w})$, because by construction, $\bar{w}_f$ is not a multiple of $k$.

Now, $v+\lambda w$ is an $\R$-flow on $X$ for any $\lambda\in\R$. Because $w_f=0$ for all $f\in\zero(\bar{w})$, we can choose $\lambda$ such that $v+\lambda w$ is nowhere zero and has the same sign as $v$ on all facets $f\in\zero(\bar{w})$. By Proposition~\ref{bijectionflowregions},
there exists a totally cyclic orientation $\ori$ compatible with $v + \lambda w$, and we have $\ori|_{\zero(\bar{w})}=\sigma$ as desired.
\end{proof}

\subsection{Modular Tension Reciprocity} \label{subsec:mod-ten-rec}

Recall that the modular tension function
$\tau^*_X(k)$ gives the number of nowhere-zero $\Z_k$-tensions of $X$,
or equivalently (by Proposition~\ref{cuts-equal-tensions-mod-n})
the number of nowhere-zero $\Z_k$-cuts.  In this section we again
use lattice-point methods to obtain a reciprocity theorem for
modular tensions.  The theorems and proofs are similar, and in many
places dual, to those of Section~\ref{subsec:mod-flow-rec}.

Let $\SS$ be a set of vectors that forms a basis for $\ker\bd$ as a $\Zz$-module. Let
$[\SS]$ be the matrix whose rows are the vectors in $\SS$. Then the tension space
\[
 \Ten_\R(X) = \left\{ \psi\in\R^F \; \middle| \; [\SS]\psi = 0 \right\}
\]
is of dimension $\rho$.
For each $b\in\Z^R$, define
\[Q_\SS(b):=\left\{\psi\in[0,1]^F:[\SS]\psi=b\right\},\qquad
Q_\SS^\circ(b):=\left\{\psi\in(0,1)^F:[\SS]\psi=b\right\}.
\]
Call $b$ \emph{feasible} if $Q_\SS^\circ(b)\neq\0$.  Let $\BBB(\SS)$
denote the set of all feasible vectors, and define
\[
Q_\SS:=\bigsqcup_{b\in \BBB(\SS)}Q_\SS(b) \, ,\qquad
Q_\SS^\circ:=\bigsqcup_{b\in \BBB(\SS)}Q_\SS^\circ(b) \, .
\]
Then the modular tension function of $X$ is the Ehrhart function of $Q^\circ_\SS$, that is, 
\[
\tau^*_X(k)= E_{Q^\circ_\SS}(k)=\left|\Z^F\cap k Q_\SS^\circ\right|.
\]
For each $b\in\BBB$, the dimension of $Q^\circ_\SS(b)$ is $\rho$, since the intersection of a translate of $\Ten_\R(X)$ with the open set $(0,1)^F$ has the same dimension as $\Ten_\R(X)$. Therefore we can apply Ehrhart--Macdonald reciprocity to obtain
\begin{align*}
(-1)^{\rho} E_{Q^\circ_\SS}(-k)
&= \sum_{b\in\BBB} (-1)^{\rho} E_{Q^\circ_\SS(b)}(-k) \\
&= \sum_{b\in\BBB}  E_{Q_\SS(b)}(k) \\
&=  E_{Q_\SS}(k) 
\end{align*}
which gives us
\begin{equation}
\label{eqn:mod-tension-reciprocity}
(-1)^{\rho} \tau^*_X(-k)=\left|\Z^F\cap kQ_\SS\right|.
\end{equation}
As before, for $\bar{\psi}\in\Z_k^F$, let $\zero(\bar{\psi})$ denote the set of all $f\in F$ with $\bar{\psi}_f=0$.

The construction of $Q$ is independent of the choice of $\SS$,
for the following reason.  If $\SS'$ is another $\Zz$-basis
for $\ker\bd$, then there is an invertible $\Zz$-matrix
$C$ of size $m-\rho$ such that $C[\SS]=[\SS']$, and a
calculation similar to \eqref{change-of-basis-flow}
shows that 
\begin{align*}
Q_{C\SS}^\circ( b)&=Q_{\SS}^\circ(C^{-1}b) \, ,&
Q_{C\SS}^\circ(Ca)&=Q_{\SS}^\circ(a) \, ,\\
 \BBB(\SS) &= C^{-1}\BBB(C\SS) \, ,&
C\BBB(\SS) &= \BBB(C\SS) \, ,\\
Q_{\SS} &= Q_{C\SS} \, , &
Q^\circ_{\SS} &= Q^\circ_{C\SS} \, .
\end{align*}
For this reason, henceforth we write $Q$ instead of $Q_\SS$.

If~$X$ has no loops---i.e., $\bd$ has no zero columns---then a
generic real linear combination of its rows gives a nowhere-zero
vector in $\Cut_\Rr(X)$, hence in $\Ten_\Rr(X)$. Thus, if $X$ has no
loops, then $\Ten_\R(x)$ is not contained in any coordinate
hyperplane.

To state the modular analogue of Theorem \ref{integral-tension-reciprocity}, we define
\begin{equation} \label{tension-pair}
\Psi_k(X) := \left\{ (\bar{\psi},\sigma) \st
\begin{array}{l}
\text{$\bar{\psi}$ is a $\Z_k$-tension on $X$ and}\\
\text{$\sigma:\zero(\bar{\psi})\rightarrow\{-1,1\}$ extends}\\
\text{to an acyclic orientation of~$X$}
\end{array}
\right\}.
\end{equation}

\begin{theorem} \label{modular-tension-reciprocity}
Suppose that $X$ has no loops.  Then 
\[
(-1)^{\rho}\tau^*_X(-k)=|\Psi_k(X)| \, .
\]
\end{theorem}

\begin{proof}
For $\psi\in\Zz^F\cap k Q$, define $h(\psi)=(\bar{\psi},\sigma)$, where $\bar{\psi}=\psi \bmod k$ and
\[
\sigma_f :=
\begin{cases}
-1 & \text{ if } \psi_f = 0,\\
1 & \text{ if } \psi_f = k.
\end{cases}
\]
We will show that $h$ is a bijection. By \eqref{eqn:mod-tension-reciprocity} this will complete the proof.

First, to show that $h$ is well-defined, we observe on the one hand that, by construction, $\bar{\psi}$ is a $\Z_k$-tension on $X$ because $\bd\bar{\psi}=kb'$ for some $b'\in\Zz^R$. On the other hand, an acyclic orientation $\ori$ of $X$ with $\sigma = \ori|_{\zero(\bar{\psi})}$ can be constructed as follows. Let $b\in\Zz^R$ such that $\psi\in kQ(b)$. By construction
\[
kQ^\circ(b) = (\Ten_\R(X) + \psi)\cap (0,k)^F.
\]
In particular, $kQ^\circ(b)$ is a relatively open, non-empty set of dimension $\rho$.  Since $X$ has no loops, 
the tension space is not contained in any coordinate hyperplane,
and the intersection of any coordinate hyperplane with $\Ten_\R(X)$ has dimension $\rho-1$.  We have
\[
kQ^\circ(b)\sm (\BB + \psi) = ((\Ten_\R(X) \sm\BB) + \psi)\cap (0,k)^F,
\]
and so $kQ^\circ(b)\sm (\BB + \psi)$ can be obtained from $kQ^\circ(b)$ by removing finitely many sets of dimension $\rho-1$. This implies in particular that $kQ^\circ(b)\sm (\BB + \psi)$ is non-empty. Let $\hat{\psi}\in kP^\circ(b)\sm (\BB + \psi)$ and $v:=\psi-\hat{\psi}$. As both $\psi,\hat{\psi}\in\Ten_\R(X)+w$, it follows that $v\in\Ten_\R(X)$. As $\hat{\psi}\not\in(\BB+\psi)$, it follows that $v\not\in\BB$. Thus, $v$ is a nowhere-zero $\R$-tension on $X$. Let $\ori$ denote the sign vector of $v$. By the reciprocity theorem for integral tensions (Theorem~\ref{integral-tension-reciprocity}) we know that $\ori$ is an acyclic orientation of $X$. Moreover, for $f\in\zero(\bar{\psi})$, either $\psi_f=0$, when $\hat{\psi}_f>\psi_f$, or else $\psi_f=k$, when $\hat{\psi}_f < \psi_f$, and in either case $\ori_f = \sigma_f$.

It is immediate from the definition that $h$ is injective.
To show surjectivity, suppose that $\bar{\psi}$ is a
$\Z_k$-tension on $X$ and $\sigma$ is a function
$\zero(\bar{\psi})\to\{-1,1\}$ such that there exists a totally cyclic
orientation $\ori$ of $X$ with $\sigma=\ori|_{\zero(\bar{\psi})}$.  
Define $\psi\in\Zz^F$ by
\[
\psi_f := \begin{cases}
\lift(\bar{\psi}_f) & \text{if $f\not\in\zero(\bar{\psi})$},\\
0 & \text{if $f\in\zero(\bar{\psi})$ and $\sigma_f=-1$},\\
k & \text{if $f\in\zero(\bar{\psi})$ and $\sigma_f=1$}.
\end{cases}
\]
Then $[\SS]\psi = kb$ for some $b\in\Z^R$
(because $\bar{\psi}$ is a $\Zz_k$-tension) and $\psi\in\Z^F\cap[0,k]^F$. So
all that remains to show is that $Q^\circ(b)\neq\emptyset$.  By
Theorem~\ref{integral-tension-reciprocity}, there exists a nowhere-zero
real tension $v$ with $\sgn(v)=\ori$; we may choose $v$ to have
arbitrarily small magnitude.  
We claim that $\hat{\psi}=\psi-v\in
Q^\circ(b)$.  Indeed, first, we have
$[\SS]\hat{\psi} = [\SS]\psi - [\SS] v = kb - 0  = kb$.
Second, if $f\not\in\zero(\bar{\psi})$ then $0<\hat{\psi}_f <k$.
If $\psi_f = 0$ then $\sgn(v_f)=-1$ and  $\hat{\psi}_f > 0$.
while if $\psi_f = k$ then $\sgn(v_f)=1$ and $\hat{\psi}_f < k$.
This proves the claim and completes the proof that $h$ is surjective.
\end{proof}

\begin{corollary} \label{acyclic-orientations-from-tensions}
If $X$ has no loops, then $|\tau^*_X(-1)|$
is the number of acyclic orientations of $X$.
\end{corollary}
\begin{proof}
If we set $k=1$ in the definition \eqref{tension-pair}, we see that
the vector $\bar{\psi}$ must be the zero vector, and $\sigma=\ori$ is
an acyclic orientation of~$X$.
\end{proof}

\begin{corollary}
\label{modular-tension-reciprocity-for-shrinkable-complexes}
Let $X$ be a totally unimodular complex without loops. Then $|\tau^*_X(-k)|$ counts the number of pairs $(\bar{w},\sigma)$ where $\bar{w}$ is a $\Z_k$-tension on $X$ and $\sigma$ is an acyclic orientation on $X\setminus\supp(\bar{w})$.
\end{corollary}

\begin{proof} To derive the corollary from Theorem~\ref{modular-tension-reciprocity}, all we have to show is that for any $\Z_k$-tension $\bar{w}$ and any $\sigma:\zero(\bar{w})\rightarrow\{-1,1\}$ the following statements are equivalent:
\begin{enumerate}[(1)]
\item $\sigma$ extends to an acyclic orientation of $X$. \label{enum:delete-support-1}
\item $\sigma$ is an acyclic orientation of $X\setminus\supp(\bar{w})$. \label{enum:delete-support-2}
\end{enumerate}

Suppose that $\supp(\bar w)=\{f_1,\dots,f_\ell\}$.
Define $X_{i}:=X\setminus\{f_1,\dots,f_{i}\}$ for $i=0,\ldots,\ell$. Thus $X_0=X$ and $X_\ell = X\setminus\supp(\bar{w})$.

We will need the following two observations:
\begin{enumerate}[(a)]
\item Let $Y$ be a cell complex and let $f$ be a facet of $Y$.  If $\psi$ is an $\R$-tension on $Y$, then $\psi|_{Y\setminus f}$ is an $\R$-tension on $Y\setminus f$. Conversely, if $\psi'$ is an $\R$-tension on $Y\setminus f$, then there exists an $\R$-tension $\psi$ on $Y$ such that $\psi|_{Y\setminus f}=\psi'$.  This is easy to derive from the fact that $\psi$ is an $\R$-tension on $Y$ if and only if there exists $c\in \R^R$ such that $\psi=c\partial_Y$; see also Proposition~\ref{coloring-cut}. \label{enum:deletion}
\item If $X$ is TU, then $[\SS]$ can be chosen to be totally unimodular, by Lemma~\ref{tu-basis-of-kernel}.
\label{enum:tu}
\end{enumerate}

Now we can turn to the proof proper.

\eqref{enum:delete-support-1} \emph{implies} \eqref{enum:delete-support-2}: Let $\ori$ be an acyclic orientation of $X$ such that $\ori|_{\zero(\bar{w})}=\sigma$. Then by Proposition~\ref{bijectiontensionregions} there exists an $\ori$-positive $\R$-tension $\psi$ on $X$. We claim that $\psi|_{\zero(\bar{w})}$ is a $\sigma$-positive tension on $X\setminus\supp(\bar{w})$.  Indeed, $\psi|_{\zero(\bar{w})}$ is $\sigma$-positive, and by observation~\eqref{enum:deletion}, if $\psi$ restricts to an $\R$-tension on $X_{i-1}$ then it restricts to an $\R$-tension on $X_i$. By induction $\psi|_{\zero(\bar{w})}$ is an $\R$-tension on $X_\ell=X\setminus\supp(\bar{w})$. Thus, by Proposition~\ref{bijectiontensionregions}, we conclude that $\sigma$ is an acyclic orientation on $X\setminus\supp(\bar{w})$.

\eqref{enum:delete-support-2} \emph{implies} \eqref{enum:delete-support-1}: Let $\sigma$ be an acyclic orientation of $X\setminus\supp(\bar{w})$. Then by \ref{bijectiontensionregions} there exists a $\sigma$-positive $\R$-tension $\psi_\ell$ on $X_\ell=X\setminus\supp(\bar{w})$. By observation~\eqref{enum:deletion}, we can inductively construct $\R$-tensions $\psi_\ell,\psi_{\ell-1},\psi_0=\psi$ on $X_\ell,\ldots,X_0=X$ such that $\sgn(\psi_i|_{\zero(\bar{w})}) = \sigma$.  Note that these tensions need not be nowhere-zero, although $\psi$ is nonzero on $\zero(\bar{w})$.

Next, we construct an $\R$-tension $w$ on $X$ such that $w_f=0$ if and only if $f\in\zero(\bar{w})$. Such a $w$ can be obtained as follows. First, we interpret
$\bar{w}$ as a vector in $\Z^F$ by lifting all entries to their canonical representatives in $\Z$. Then $[\SS]\bar{w}=kb$ for some $b\in\Z^R$. Now we choose a
lattice point $z\in\Z^F$ such that $z|_{\zero(\bar{w})} = 0$ and $[\SS] z = b$. Such a lattice point exists, because $\frac{1}{k}\bar{w}$ is a rational solution to
this linear system.  Since $[\SS]$ is totally unimodular by observation~\eqref{enum:tu}, this system also has an integral solution by Lemma~\ref{integer-solution}.t Finally, we let $w = \bar{w} - kz$ whence $\bd w = 0$. Note that $w_f\not=0$ whenever $f\in\supp(\bar{w})$ because in this case $\bar{w}_f$ is not a multiple of $k$.

Now, $\psi+\lambda w$ is an $\R$-tension on $X$ for any $\lambda\in\R$. Because $w_f=0$ for all $f\in\zero(\bar{w})$, we can choose $\lambda$ such that $\psi+\lambda w$ is nowhere zero and has the same sign as $\psi$ on all facets $f\in\zero(\bar{w})$. Let $\ori$ be the acyclic orientation compatible with $\psi + \lambda w$, which exists by \ref{bijectiontensionregions}. We then have $\ori|_{\zero(\bar{w})}=\sigma$ as desired.
\end{proof}

\subsection{Modular Coloring Reciprocity} \label{subsec:mod-col-rec}

In analogy to the definition of the set $\Psi_k(X)$ in~\eqref{tension-pair}, let
\begin{equation} \label{coloring-pair}
C_k(X) := \left\{ (c,\sigma) \st
\begin{array}{l}
\text{$c$ is a $\Z_k$-coloring on $X$ and}\\
\text{$\sigma:\zero(c\bd)\rightarrow\{-1,1\}$ extends}\\
\text{to an acyclic orientation of~$X$}
\end{array}
\right\}.
\end{equation}

\begin{theorem}
\label{modular-coloring-reciprocity}
Suppose $X$ has no loops. Then 
$$(-1)^n\chi^*_X(-k)=|C_k(X)| \, .$$
Moreover, $|\chi^*_X(-1)|$ equals the number of acyclic orientations of $X$. 
\end{theorem}

\begin{proof} Consider the map $f:C_k(X)\to\Psi_k(X)$
defined by $f(c,\sigma)=(c\bd, \sigma)$.
This map is surjective because every tension is a cut over $\Zz_k$ (Proposition~\ref{cuts-equal-tensions-mod-n})
and the cut space is by definition the rowspace of $\bd$.  By \eqref{count-kernel}, for each $\psi\in\Ten_k(X)$ we have
\[|f^{-1}(\psi)| = |\ker_{\Zz_k}(\cbd)| = k^{n-\rho} \gamma(\bd,k) \, .\]
Therefore,
\begin{align*}
|C_k(X)| &= k^{n-\rho} \gamma(\bd,k) |\Psi_k(X)|\\
&= (-1)^\rho k^{n-\rho} \gamma(\bd,k) \tau^*_X(-k) && \text{(by Thm.~\ref{modular-tension-reciprocity})}\\
&= (-1)^n \chi^*_X(-k) && \text{(by Thm.~\ref{modular-cut-theorem})}.
\end{align*}
The second assertion now follows from Corollary~\ref{acyclic-orientations-from-tensions}.
\end{proof}

\section{Further Directions}
We conclude with some possibilities for future research.

\begin{enumerate}
\item We do not know whether there exists a non-SQU cell complex~$X$ whose modular chromatic function
$\chi_X^*(k)$ is an honest polynomial (that is, a quasipolynomial
with period~1).  This would be an instance of ``period collapse'';
see, e.g., \cite{Period}.  For this
to occur, all non-monomial
terms in Theorem~\ref{coloring-inc-exc} would have to cancel out.

\item For a graph $G$, the modular flow polynomial $\varphi^*_G(k)$ counts the
  number of flows not merely over $\Zz_k$, but over any Abelian group
  $A$ of cardinality~$k$.  This is easily seen to be false for general
  cell complexes; the matrix $\bd=[2]$ has two flows over $\Zz_4$
  but four over $\Zz_2\oplus\Zz_2$.  What if we impose a unimodularity
  condition on $X$?

\item Kook, Reiner and Stanton~\cite{KRS} gave a formula for the Tutte
  polynomial of a matroid as a convolution of tension and flow
  polynomials.  Breuer and Sanyal~\cite[\S5]{breuersanyal} used the
  Kook--Reiner--Stanton formula, together with reciprocity results
  to give a general combinatorial interpretation of the values of
  the Tutte polynomial of a graph~$G$ at positive integers; see also
  \cite{Interp} and \cite[Theorem~3.11.7]{B2009}.
  Do these results generalize to cell complexes whose tension
  and flow functions are polynomials?
  
\item In the graph case, the geometric setup has proven extremely useful for establishing bounds on the coefficients of the
chromatic polynomial \cite{hershswartz} and the tension and flow polynomials, in particular in the modular case
\cite{BDall2011}. Moreover these geometric constructions are closely related to Steingr\'imsson's coloring
complex~\cite{BDall2010,steingrimsson}. Can these methods be extended to the case of counting quasipolynomials defined in terms of cell complexes? In particular, what are good bounds on the coefficients of the quasipolynomials introduced in this paper?

\end{enumerate}

\bibliographystyle{amsplain}
\bibliography{biblio}

\end{document}